\documentclass[11pt]{article}
\usepackage[letterpaper,textwidth=6.5in,textheight=9in,
            centering,ignorehead,nomarginpar]{geometry}
\usepackage[T1]{fontenc}
\usepackage{textcomp}
\usepackage{xcolor}

\usepackage[normalem]{ulem}
\usepackage{url}
\usepackage{hyperref}
\usepackage{doi}
\usepackage{appendix}
\usepackage{array}

\usepackage[sort&compress]{natbib}

\bibpunct{(}{)}{;}{a}{,}{;}
\usepackage{graphicx}
\usepackage{bbm}
\usepackage{amsmath,amsthm,amssymb,amsfonts,latexsym}
\usepackage{mathtools}
\usepackage{booktabs}
\usepackage{icomma}
\usepackage{afterpage}
\usepackage{titling}
\usepackage{bm}
\thanksmarkseries{alph}
\usepackage[labelfont={small,sc},textfont={small,it},
            margin=15pt,skip=10pt,position=bottom]{caption}
\usepackage[labelfont={scriptsize,normalfont},textfont={scriptsize,it},
            margin=20pt,skip=5pt,position=bottom,
            labelformat=simple]{subcaption}

\usepackage{algorithmic}
\usepackage[boxed]{algorithm}
\usepackage{enumitem}

\newtheorem{remark}{Remark}

\newtheorem{proposition}{Proposition}

\numberwithin{condition}{section}
\numberwithin{assumption}{section}
\numberwithin{remark}{section}
\numberwithin{equation}{section}
\numberwithin{lemma}{section}
\numberwithin{definition}{section}
\numberwithin{theorem}{section}
\numberwithin{proposition}{section}
\numberwithin{table}{section}
\numberwithin{figure}{section}
\numberwithin{theorem}{section}
\numberwithin{corollary}{section}
\numberwithin{property}{section}
\numberwithin{algorithm}{section}

\newcommand{\EQ}{\begin{equation}}
\newcommand{\EN}{\end{equation}}
\newcommand{\EQS}{\begin{equation*}}
\newcommand{\ENS}{\end{equation*}}
\newcommand{\ds}{\displaystyle}

\newcommand{\mycomment}[1]{}

\newcommand*\xbar[1]{%
  \hbox{%
    \vbox{%
      \hrule height 0.5pt
      \kern0.5ex%
      \hbox{%
        \kern-0.1em%
        \ensuremath{#1}%
        \kern-0.1em%
      }%
    }%
  }%
}

\def\n1{n}
\def\argmax{\mathop{\rm arg\,max}}

\newsavebox{\savepar}

\numberwithin{equation}{section}
\numberwithin{table}{section}
\numberwithin{figure}{section}

\def\argmax{\mathop{\rm arg\,max}}

\usepackage[running]{lineno}

\begin{document}
\title{ Machine Learning and Hamilton-Jacobi-Bellman Equation  \\
    for Optimal Decumulation:  a Comparison Study \\
         {\em } 
}

\author{
    Marc Chen \thanks{David R. Cheriton School of Computer Science,
        University of Waterloo, Waterloo ON, Canada N2L 3G1,
          \texttt{marcandre.chen@uwaterloo.ca}}
    \and
    Mohammad  Shirazi  \thanks{David R. Cheriton School of Computer Science,
        University of Waterloo, Waterloo ON, Canada N2L 3G1,
         \texttt{mmkshirazi@uwaterloo.ca}}
    \and
   Peter A. Forsyth\thanks{David R. Cheriton School of Computer Science,
        University of Waterloo, Waterloo ON, Canada N2L 3G1,
        \texttt{paforsyt@uwaterloo.ca}}
       \and Yuying Li \thanks{David R. Cheriton School of Computer Science,
        University of Waterloo, Waterloo ON, Canada N2L 3G1,
        \texttt{yuying@uwaterloo.ca}}
  }
\maketitle


\begin{abstract}
We propose a novel data-driven neural network (NN) optimization framework for solving
an optimal 
stochastic control problem under stochastic constraints. 
Customized activation functions for the
output layers of the NN are applied, which permits training via  standard unconstrained optimization. 
The optimal solution yields a multi-period asset allocation and decumulation strategy for a holder of a defined contribution (DC) pension plan. The objective function of the optimal control problem is based on expected wealth withdrawn (EW) and expected shortfall (ES) that directly targets left-tail risk. The stochastic bound constraints 
enforce a guaranteed minimum withdrawal each year. 
We demonstrate that the data-driven approach is capable of learning a near-optimal solution by benchmarking it  against the numerical results from a Hamilton-Jacobi-Bellman (HJB)  Partial Differential Equation (PDE) computational framework.

\vspace{5pt}


\end{abstract}

\section{Introduction}

Access to traditional defined benefit (DB) pension plans continues to disappear 
for employees. In 2022, only 15\% of private sector workers in 
the United States had access to a defined benefit plan, 
while 66\% had access to a defined contribution (DC) plan 
\citep{bls_statistics_2023}. In other countries, DB plans have become a thing of the past.

Defined contribution plans leave the burden of creating an investment and withdrawal strategy to the individual investor, which Nobel Laureate William Sharpe referred to as ``the nastiest, hardest problem in finance'' \citep{Sharpe2017}.  Indeed, a review of the literature on decumulation strategies \citep{bernhardt-donnelly:2018, MacDonald2013} shows that balancing all of retirees' concerns with a single strategy is exceedingly difficult. To address these concerns and find an optimal balance between maximizing withdrawals and minimizing the risk of depletion, while guaranteeing a minimum withdrawal, the approach in \citet{forsyth:2022} determines a decumulation and allocation strategy for a standard 30-year investment horizon by formulating it as a problem in optimal stochastic control.  Numerical solutions are
obtained using dynamic programming, which results in a 
Hamilton-Jacobi-Bellman (HJB) Partial Differential Equation (PDE).

The HJB PDE framework developed in \citet{forsyth:2022} maximizes
expected withdrawals and minimizes the risk of running out of savings, measured by the 
left-tail in the terminal wealth distribution.  Since maximizing withdrawals and minimizing risk
are conflicting measures, we use a scalarization technique to compute Pareto optimal points.
A  constant lower bound is imposed on the withdrawal, providing a guaranteed income.  
An upper bound on withdrawal is also imposed,    which can be viewed
as the target withdrawal.
The constraints of no shorting and no leverage are imposed on the investment allocation.

The solution to this  constrained stochastic optimal control problem yields a dynamic stochastic strategy,  
naturally aligning with  retirees' concerns and objectives.
Note that cash flows are   not  mortality weighted, consistent with \citet{Bengen1994}.
This can be justified on the basis of {\em planning to live, not planning
to die} as discussed in \cite{Pfau_2018}.

Our dynamic strategy can be contrasted to traditional strategies 
such as the \emph{Bengen Rule} (4\% Rule), which recommends withdrawing a 
constant 4\% of initial capital each year (adjusted for inflation) 
and investing equal amounts into stocks and bonds \citep{Bengen1994}. Initially proposed in 1994, \citet{scott-jason-sharpe} found the 4\% Rule to 
still be a popular strategy 14 years later, and the near-universal recommendation of the top brokerage and retirement planning groups. Recently there has been acknowledgment 
in the asset management industry that the 4\% Rule is sub-optimal, but wealth
managers still recommend variations of the same constant withdrawal principle \citep{williams_kawashima}. 
The strategy proposed by \citet{forsyth:2022} is shown to be far more efficient than the Bengen 4\% Rule.
Of course, the PDE solution in \citet{forsyth:2022}                                     
is restricted to low dimensions (i.e. a small number of stochastic factors).

In order to remedy some of the deficiencies of PDE methods (such
as in \citet{forsyth:2022}),  
we propose a neural network (NN) based framework without using dynamic programming.
In contrast to the PDE solution approach, our proposed NN approach has the following advantages: 

\begin{enumerate} [label = (\roman*)]
    \item It is data-driven and does not depend on a 
       parametric model. This makes the framework versatile in selecting 
      training data, and less susceptible to model misspecification. 

    \item The control is learned directly, thereby 
     exploiting the low dimensionality of the control \citep{beating_benchmark}. This technique
     thus avoids dynamic programming and the associated error propagation. 
      The NN approach can also be applied to higher dimensional 
       problems, such as those with a large number of assets.
    
    \item If the control is a continuous function of time, the 
     control approximated by the NN framework will reflect this property. 
           If the control is discontinuous \footnote{{\em{Bang-bang}} controls, 
          frequently encountered in optimal control, are discontinuous as a function of time.}, the 
               NN seems to produce a smooth, but quite accurate, 
               approximation.\footnote{For a possible explanation
                of this, see \citet{Ismailov}.} 

\end{enumerate}

Since the  NN only generates an approximate solution to the complex  
stochastic optimal control problem,
  it is essential to  assess 
its accuracy and robustness. 
Rarely is the quality of an NN solution assessed rigorously, since an accurate
solution to the optimal control problem is often not available. 
In this paper, we compare the NN solution to the 
decumulation problem against the ground-truth results 
from the provably convergent HJB PDE method. 

We have previously seen such a comparison in different applications,  
see, e.g.,  \citet{lauriere2021performance} for a comparison study on a fishing control problem.  
As machine learning and artificial intelligence based methods continue to proliferate 
in finance and investment management, it is crucial to demonstrate 
that these methods are reliable and explainable \citep{imf_ai_finance}. We believe that 
our proposed framework and test results make a step forward in demonstrating 
deep learning's potential  for stochastic control problems in finance.

To summarize, the main contributions of this paper are as follows:

\begin{itemize}
\item Proposing an NN framework with suitable activation functions for 
decumulation and allocation controls, which yields 
an approximate solution to the constrained stochastic optimal 
decumulation problem  in \citet{forsyth:2022} by solving a standard unconstrained optimization problem;


\item
Demonstrating  that the NN solution achieves very high accuracy in terms of
the efficient frontier and the decumulation control when compared to the solution from the provably convergent HJB PDE method;

\item
Illustrating that, with a suitably small regularization parameter, the NN allocation strategy can differ significantly from the 
PDE allocation strategy  in the region of high wealth and 
near the terminal time, while  the  relevant performance
statistics remain unaffected. This is due to the fact that the problem
is ill-posed in these regions of the state space unless we add
a small regularization term;

\item 
Testing the NN solution's robustness on out-of-sample and out-of-distribution data, as well as its versatility in using different datasets for training.

\end{itemize}

While other neural network and deep learning methods for optimal stochastic control  problems have been proposed before,  they differ significantly from our approach in their architecture, taking a {\em stacked} neural network approach as in \citet{buehler2019deep, han_weinan_stack, tsang_wong_dpl} or a hybrid dynamic programming and reinforcement learning approach \citep{hure_nn_rl}. On the other hand, our framework uses the same two neural networks at all rebalancing times in the investment scenario. 
Since our NNs take time as an input, the solution will be continuous in time 
if the control is continuous.  Note that the idea of using time
as an input to the NN was also suggested in \citet{lauriere2021performance}.  
According to the taxonomy of sequential decision problems proposed in \citet{Powell2021}, our approach would most closely be described as Policy Function Approximation (PFA).

With the exception
of \citet{lauriere2021performance}, previous
papers do not provide a benchmark for numerical methods, as we do
in this work. Our results show that our proposed NN method 
is able to approximate the numerical results in \citet{forsyth:2022} with high accuracy. 
Especially notable, and somewhat unexpected, is that the {\em bang-bang}
control\footnote{In optimal stochastic control, a bang-bang control is a discontinuous
function of the state.} for the withdrawal is reproduced very closely with the NN method.
\section{Problem Formulation}

\subsection{Overview}

The investment scenario described in \citet{forsyth:2022} concerns an investor with a  portfolio wealth of a specified size,
upon retirement. 
The investment horizon is fixed with a finite number of 
equally spaced rebalancing times (usually annually). 
At each rebalancing time, the investor first chooses how much 
to withdraw from the portfolio and then how to allocate the remaining wealth. 
The investor must withdraw an amount within a specified range. 
The wealth in this portfolio can be allocated to any mix of two given assets, with no shorting or leverage. The assets the investor can access are a broad stock index fund and a constant maturity bond index fund. 

In the time that elapses between re-balancing times, the 
portfolio's wealth will change according to the dynamics of the underlying assets. 
If the wealth of the portfolio  goes below zero (due to minimum withdrawals), the portfolio is liquidated, trading 
ceases, debt accumulates at the borrowing rate, and withdrawals are restricted to the minimum amount. 
At the end of the time horizon, a final withdrawal 
is made and the portfolio is liquidated, yielding the terminal wealth. 

We assume here that the investor has other assets, 
such as real estate, which are non-fungible with investment assets. These other
assets can be regarded as a hedge of last resort, which can be used to fund any accumulated debt 
\citep{Pfeiffer_2013}.  
This is not a novel assumption and is in line with the mental bucketing 
idea proposed by \citet{shefrin-thaler2}. The use of this assumption 
within literature targeting similar problems is also common (see \citet{forsyth2022tontine}). Of course, the objective of the optimal control is to make running out of savings an unlikely event.

The investor's goal then is to maximize the weighted sum of total withdrawals and the 
mean of  the worst 5\% of the outcomes (in terms of terminal wealth).
We term this tail risk measure as Expected Shortfall (ES) at the 5\% level.  
In this section, this optimization problem will be described with the mathematical details 
common to both the HJB and NN methods. 

\subsection{Stochastic Process Model}
\label{Market Model Section}
Let $S_t$ and $B_t$ represent the real (i.e. inflation-adjusted) \emph{amounts} invested in the stock index and a constant
maturity bond index, respectively. These assets are modeled with correlated jump diffusion models,  in line 
with \citet{mitchell_2014}. These parametric
stochastic differential equations (SDEs)  allow us to model non-normal asset returns. The SDEs are used in solving the HJB PDE, and  generating training data with Monte Carlo simulations in the proposed NN framework. For the remainder of this paper, we refer to simulated data using these models as {\em synthetic} data.   

When a jump occurs, $S_t = \xi^s S_{t^-}$, where $\xi^s$ is a random number representing the jump multiplier and $S_{t^-} =   S(t - \epsilon), \epsilon \rightarrow 0^+$ ($S_{t^-}$ is the instant of time before $t$). We assume that $\log(\xi^s)$ follows a double exponential distribution \citep{kou:2002,Kou2004}. The jump is either upward or downward, with probabilities $u^s$ and $1-u^s$ respectively. The density function for $y = \log (\xi^s)$ is

{\color{black}
\begin{linenomath*}
\begin{equation}
f^s(y) = u^s \eta_1^s e^{-\eta_1^s y} {\bf{1}}_{y \geq 0} +
       (1-u^s) \eta_2^s e^{\eta_2^s y} {\bf{1}}_{y < 0}~.
\label{eq:dist_stock}
\end{equation}
\end{linenomath*}
}

We also define
{\color{black}
\begin{linenomath*}
\begin{eqnarray}
\gamma^s_{\xi} &= &E[ \xi^s -1 ]
          =   \frac{u^s \eta_1^s}{\eta_1^s - 1} + 
          \frac{ ( 1 - u^s ) \eta_2^s }{\eta_2^s + 1}  -1 ~.
\end{eqnarray}
\end{linenomath*}
}

The starting point for building the jump diffusion model is a standard geometric Brownian motion, with drift rate $\mu^s$ and volatility $\sigma^s$. A third term is added to represent the effect of jumps, and a 
compensator is added to the drift term to preserve the expected drift rate. For stocks, this gives the 
following stochastic differential equation (SDE) that describes how $S_t$ (inflation adjusted)
evolves in the absence of a control:

\begin{eqnarray}
  \frac{dS_t}{S_{t^-}} &= &\left(\mu^s -\lambda_\xi^s \gamma_{\xi}^s \right) \, dt + 
    \sigma^s \, d Z^s +  d\left( \ds \sum_{i=1}^{\pi_t^s} (\xi_i^s -1) \right) ,
  \label{jump_process_stock}
\end{eqnarray}
where $d Z^s$ is the increment of a Wiener process,
$\pi_t^s$ is a Poisson process with positive intensity parameter $\lambda_\xi^s$, and $\xi_i^s ~\forall i$ are i.i.d.\ positive random variables having distribution (\ref{eq:dist_stock}).
Moreover, $\xi_i^s$, $\pi_t^s$, and $Z^s$ are assumed to all be mutually independent.

As is common in the practitioner literature, we directly model the returns
of a constant maturity (inflation adjusted) bond index by a jump diffusion process \citep{mitchell_2014,Lin_2015}.
Let the amount in the constant maturity bond index be $B_{t^-} =  B(t - \epsilon), \epsilon \rightarrow 0^+$.
In the absence of a control, $B_t$ evolves as
\begin{eqnarray}
\frac{dB_t}{B_{t^-}} &= &\left(\mu^b -\lambda_\xi^b \gamma_{\xi}^b  
   + \mu_c^b {\bf{1}}_{\{B_{t^-} < 0\}}  \right) \, dt + 
  \sigma^b \, d Z^b +  d\left( \ds \sum_{i=1}^{\pi_t^b} (\xi_i^b -1) \right) ,
\label{jump_process_bond}
\end{eqnarray}

\noindent where the terms in Equation (\ref{jump_process_bond}) are defined analogously to Equation (\ref{jump_process_stock}). In particular, $\pi_t^b$ 
is a Poisson process with positive intensity parameter
$\lambda_\xi^b$, $\gamma_{\xi}^b = E[ \xi^b -1 ]$, and $y=\log(\xi^b)$ has the same 
distribution as in 
equation (\ref{eq:dist_stock}) (denoted by $f^b(y)$) 
with distinct parameters, $u^b$, $\eta_1^b$, and $\eta_2^b$. Note that $\xi_i^b$, $\pi_t^b$, 
and $Z^b$ are assumed to all be mutually independent, as in the stock SDE. 
The term $\mu_c^b {\bf{1}}_{\{B_{t^-} < 0\}}$ represents the extra cost of borrowing (a spread). 

The correlation between the two assets' diffusion processes is $\rho_{sb}$, giving us $d Z^s \cdot d Z^b = \rho_{sb}~ dt$. 
The jump processes are assumed to be independent. 
For further details concerning the justification of this market model, refer to~\cite{forsyth:2022}. 

We define the investor's total wealth at time $t$ as 

\begin{equation}
  \text{Total wealth } \equiv W_t = S_t + B_t.
\end{equation}

Barring insolvency, shorting stock and using leverage (i.e., borrowing) are not permitted, a 
realistic constraint in the context of DC retirement plans. Furthermore, if the wealth ever goes below zero,
due to the guaranteed withdrawals, 
the portfolio is liquidated, trading ceases, and debt accumulates at the borrowing rate.
We emphasize that we are assuming that the retiree has other assets (i.e., residential real estate) which
can be used to fund any accumulated debt.  In practice, this could be done using
a reverse mortgage \citep{Pfeiffer_2013}.

\subsection{Notational Conventions}

We define the finite set of discrete withdrawal/rebalancing times $\mathcal{T}$,

 \begin{eqnarray}
   \mathcal{T} = \{t_0=0 <t_1 <t_2< \ldots <t_{M}=T\}  \label{T_def} ~.
\end{eqnarray}
The beginning of the investment period is $t_0 = 0$. We assume each rebalancing time is evenly spaced, meaning $t_i - t_{i-1} = \Delta t =T/M$ is constant. To avoid subscript clutter in the following, we will occasionally use the notation $S_t \equiv S(t), B_t \equiv B(t)$ and $W_t \equiv W(t)$. At each rebalancing time, $t_i \in \mathcal{T}$, the investor first withdraws an amount of cash $q_i$ from the portfolio and then rebalances the portfolio. At time $T$, there is one final withdrawal, $q_T$, and then the portfolio is liquidated. We assume no taxes are incurred on rebalancing, 
which is reasonable since retirement accounts are typically tax-advantaged. In addition, since trading is infrequent, we assume transaction costs to be negligible \citep{dang-forsyth:2014a}. Given an arbitrary time-dependent function, $f(t)$, we will use the shorthand  

\begin{eqnarray}
  f(t_i^+) \equiv \displaystyle  \lim_{\epsilon \rightarrow 0^+}
          f(t_i + \epsilon)~, ~~& & ~~
      f(t_i^-) \equiv \displaystyle  \lim_{\epsilon \rightarrow 0^+}
          f(t_i - \epsilon)  ~~.
\end{eqnarray}

The multidimensional controlled underlying process is denoted by $X\left(t\right)=
\left( S \left( t \right), B\left( t \right) \right)$,
$t\in\left[0,T\right]$. For the realized state of the system, $x = (s, b)$. 

At the beginning of each rebalancing time $t_i$, the investor withdraws the amount $q_i(\cdot)$, determined by the control at time $t_i$; that is, $q_i(\cdot) = q_i(X(t_i^-)) = q(X(t_i^-), t_i)$. This control 
is used to evolve the investment portfolio from $W_t^-$ to $W_t^+$ 
\begin{eqnarray}
  W(t_i^+) = W(t_i^-) - q_i~, & & t_i \in \mathcal{T} ~.
\end{eqnarray}
Formally, both withdrawal and allocation controls depend on the state of the portfolio before withdrawal, $X(t_i^-)$, but it will be computationally convenient to consider the allocation control as a function of the state after withdrawal since the portfolio allocation is rebalanced after the withdrawal has occurred. Hence, the allocation control at time $t_i$ is $p_i(\cdot) = p_i(X(t_i^+)) = p(X(t_i^+), t_i)$. 

\begin{eqnarray}
  p_i( X(t_i^+)) = p(X(t_i^+), t_i) = \frac{S(t_i^+)}{S(t_i^+) + B(t_i^+)} ~.
\end{eqnarray}
As formulated, the controls depend on wealth only (see \citet{forsyth:2022}
for a proof, assuming no transaction costs). 
Therefore, we make another notational adjustment for the sake
of simplicity and consider $q_i(\cdot)$ to be a function of wealth before withdrawal, $W_i^-$, and $p_i(\cdot)$ to be a function of wealth after withdrawal, $W_i^+$. 

We assume instantaneous rebalancing, which means there are no changes in asset prices in the interval $(t_i^-,t_i^+)$. A control at time $t_i$ is therefore described by a pair $( q_i(\cdot), p_i(\cdot) ) \in $ $\mathcal{Z}(W_i^-,W_i^+, t_i)$, where $\mathcal{Z}(W_i^-,W_i^+, t_i)$ represents the set of admissible control values
for $t_i$. The constraints on the allocation control are no shorting, no leverage (assuming solvency). There are minimum and maximum values for the withdrawal. When wealth goes below zero due to withdrawals ($W_i^+ <0$), trading ceases with debt accumulating at the borrowing rate, and withdrawals are restricted to the minimum. Stock assets are liquidated at the end of the investment period. 
We can mathematically state these constraints by imposing suitable bounds on the value of the controls as follows:

\begin{eqnarray}
  \mathcal{Z}_q(W_i^-,t_i)  & = &
            \begin{cases}
              [q_{\min}, q_{\max} ] ~;~   t_i \in \mathcal{T}
              ~;~  W_i^- > q_{\max}\\
              [q_{\min}, W_i^- ] ~;~   t_i \in \mathcal{T} 
              ~;~ q_{\min} < W_i^- < q_{\max} \\
              \{q_{\min}\} ~;~   t_i \in \mathcal{T}
              ~;~ W_i^- < q_{\min}
          \end{cases}    ~,\label{Z_q_def}\\
     \mathcal{Z}_p (W_i^+,t_i) &=&
           \begin{cases}
                   [0,1] & W_i^+ > 0 ~;~ t_i \in \mathcal{T}~;~ t_i \neq t_{M} \\
                   \{0\} & W_i^+ \leq 0 ~;~ t_i \in \mathcal{T}~;~  t_i \neq t_{M} \\
                   \{0\} &  t_i= T
           \end{cases}    ~,  \label{Z_p_def} \\
           \mathcal{Z}(W_i^-, W_i^+,t_i) & = & 
                \mathcal{Z}_q(W_i^-,t_i) \times \mathcal{Z}_p (W_i^+,t_i) ~.~ 
               \label{Z_def} ~
 \end{eqnarray}

At each rebalancing time, we seek the optimal control for all possible combinations
of $(S(t), B(t))$ having the same total 
wealth \citep{forsyth:2022}.  Hence, the  controls for both withdrawal and allocation are 
formally a function of wealth and time before withdrawal 
$(W_i^-,t_i) $, but for implementation purposes it will be helpful to write 
the allocation as a function of wealth and time after withdrawal $(W_i^+,t_i)$.
The admissible  control set $\mathcal{A}$ can be written as 
\begin{eqnarray}
  \mathcal{A} = \biggl\{
                  (q_i, p_i)_{0 \leq i \leq M} : (q_i, p_i) \in \mathcal{Z}(W_i^-, W_i^+,t_i) 
                \biggr\}  ~.
                \label{P_constraint_set}
\end{eqnarray}
An admissible control $\mathcal{P} \in \mathcal{A}$, can be written as
\begin{eqnarray}
  \mathcal{P} = \{ (q_i(\cdot), p_i(\cdot) ) ~:~ i=0, \ldots, M \} ~.
  \label{control_notation_a}
\end{eqnarray}
It will sometimes be necessary to refer to the tail of the control sequence at  $[t_n, t_{n+1}, \ldots, t_{M}]$, which we define as

\begin{eqnarray}
   \mathcal{P}_n =\{ (   q_n(\cdot) , p_n(\cdot) )  \ldots, (p_{M}(\cdot), q_{M}(\cdot) ) \} ~.
\end{eqnarray}
The essence of the problem, for both the HJB and NN methods 
outlined in this paper, will be to find an optimal control $\mathcal{P}^*$.

\subsection{Risk: Expected Shortfall}

Let $g(W_T)$ be  the probability density of terminal wealth $W_T$ at $t=T$. Then
suppose
\begin{eqnarray}
\int_{-\infty}^{W^*_{\alpha}}  g(W_T) ~dW_T = \alpha~,
\end{eqnarray}
i.e.,  \emph{Pr}$[W_T <  W^*_{\alpha}] = \alpha$,  and $W^*_{\alpha}$ is the  Value at risk (VAR) 
at the level $\alpha$. We then define the Expected Shortfall (ES) as the 
mean of the worst $\alpha$ fraction of the terminal wealth. Mathematically, 
\begin{eqnarray}
{\text{ES}}_{\alpha} = \frac{\int_{-\infty}^{W^*_{\alpha}} W_T ~  g(W_T) ~dW_T }
                      {\alpha} ~. \label{ES_def_1}
\end{eqnarray}
As formulated, a higher ES is more desirable than a smaller ES (Equation (\ref{ES_def_1}) is
formulated in terms of final wealth not losses). 
It will be convenient use the alternate definition of ES
as suggested by \citet{Uryasev_2000},
\begin{eqnarray}
  {\text{ES}}_{\alpha}  & = & 
      \sup_{W^*} E
   \biggl[W^* + \frac{1}{\alpha} \min(  W_T - W^* , 0 )
                       \biggr] ~.
\end{eqnarray}
Under a control $\mathcal{P}$, and initial state $X_0$, this becomes:
\begin{eqnarray}
  {\text{ES}}_{\alpha}( X_0^-, t_0^-)  & = & 
      \sup_{W^*} E_{\mathcal{P}}^{X_0^-, t_0^-}
   \biggl[W^* + \frac{1}{\alpha} \min(  W_T - W^* , 0 ) 
                       \biggr] ~.
\end{eqnarray}
The candidate values of $W^*$ can be taken from the set of possible values of $W_T$.
It is important to note here that we define $\text{ES}_{\alpha}(X_0^-,t_0^-)$ which is the value of $\text{ES}_\alpha$ as seen at $t_0^-$. Hence, $W^*$ is fixed throughout the investment horizon.
In fact, we are considering the induced time consistent strategy, as opposed to the
time inconsistent version of an expected shortfall policy \citep{Strub_2019_a,forsyth_2019_c}.  This issue
is addressed in more detail in Appendix \ref{time_consistent_appendix}.

\subsection{Reward: Expected Total Withdrawals}

We use expected total withdrawals as a measure of reward. Mathematically, we define expected withdrawals (EW) as
\begin{eqnarray}
      {\text{EW}}( X_0^-, t_0^-) = E_{\mathcal{P}}^{X_0^-, t_0^-} 
                                 \biggl[ 
                                    \sum_{i=0}^{M} q_i
                                 \biggr]~.
                  \label{reward_eqn}
\end{eqnarray}
\begin{remark}[No discounting, no mortality weighting]
Note that we do not discount the future cash flows in Equation (\ref{reward_eqn}).  We remind the reader
that all quantities are assumed real (i.e. inflation-adjusted), so that we are effectively assuming a real
discount rate of zero, which is a conservative assumption.  This is also consistent with the approach
used in the classical work of \citet{Bengen1994}.
In addition, we do not mortality weight the cash flows, which is also 
consistent with \citet{Bengen1994}.  See \citet{Pfau_2018} for a discussion
of this approach (i.e. {\em plan to live, not plan to die}).
\end{remark}

\subsection{Defining a Common Objective Function} \label{common_obj}

In this section, we  describe the common objective function used by both the HJB method and the NN method.

Expected Withdrawals (EW) and Expected Shortfall (ES) are conflicting measures.
We use a  scalarization method to determine
Pareto optimal points for this multi-objective problem. 
For a given $\kappa$, we seek the optimal control $\mathcal{P}_{0}$    such that the following is maximized,  
\begin{eqnarray}
  {\text{EW}}( X_0^-, t_0^-)  + \kappa {\text{ES}}_{\alpha}( X_0^-, t_0^-)
      ~.  
      \label{objective_highlevel}
 \end{eqnarray}
We define (\ref{objective_highlevel}) as the pre-commitment EW-ES problem $(PCEE_{t_0}(\kappa))$ and write the problem formally as 

\begin{eqnarray}
(\mathit{PCEE_{t_0}}(\kappa)): ~~~~~~~~~~~~~~~~~~~~~~~~~~~~~~~~~~~~~~~~~~~~
~~~~~~~~~~~~~~~~~~~~~~~~~~~~~~~~~~~~~~~~~~~~~~~
\nonumber \\
    ~J\left(s,b,  t_{0}^-\right)  
     =  \sup_{\mathcal{P}_{0}\in\mathcal{A}}
          \sup_{W^*}
        \Biggl\{
               E_{\mathcal{P}_{0}}^{X_0^-,t_{0}^-}
           \Biggl[ ~\sum_{i=0}^{M} q_i ~  + ~
              \kappa \biggl( W^* + \frac{1}{\alpha} \min (W_T -W^*, 0) \biggr)
              \overbrace{+ \epsilon W_T}^{stabilization}
                    \Biggr. \Biggr. \nonumber \\
           \Biggl. \Biggl. 
                \bigg\vert X(t_0^-) = (s,b)
                   ~\Biggr] \Biggr\} \nonumber \\
    \text{ subject to } 
               \begin{cases}
(S_t, B_t) \text{ follow processes \eqref{jump_process_stock} and \eqref{jump_process_bond}};  
 ~~t \notin \mathcal{T} \\
      W_{i}^+ = S_{i}^{-} + B_{i}^{-}  - q_i \,; ~ X_i^+ = (S_i^+ , B_i^+)  \\
   S_i^+ = p_i(\cdot) W_i^+ \,; 
 ~B_i^+ = (1 - p_i(\cdot) ) W_i^+ \, \\
    ( q_i(\cdot) , p_i(\cdot) )  \in \mathcal{Z}(W_i^-, W_i^+,t_i)  \\
    i = 0, \ldots, M ~;~ t_i \in \mathcal{T}  \\
               \end{cases}  ~.
\label{PCEE_a}
\end{eqnarray}

The $\epsilon W_T$ 
stabilization term serves to avoid ill-posedness in the problem when $W_t \gg W^*$, $t \rightarrow T$,
and has little effect on optimal (ES, EW) or other summary statistics when $|\epsilon| \ll 1$. Further details about this stabilization term and its effects on both the HJB and NN framework will be discussed in Section \ref{results_section}.
The objective function in (\ref{PCEE_a}) serves as the basis for the value function in the HJB framework and the loss function for the NN method. 

\begin{remark}[Induced time consistent policy]
Note that a strategy based on $(\mathit{PCEE_{t_0}}(\kappa))$
is formally a pre-commitment strategy (i.e., not time
consistent).   However, we will assume that the retiree actually follows the induced
time consistent strategy \citep{Strub_2019_a,forsyth_2019_c,forsyth:2022}.  This control
is identical to the pre-commitment control at time zero.
See Appendix \ref{time_consistent_appendix} for more discussion of this subtle point.
In the following, we will refer to the strategy determined by (\ref{PCEE_a}) 
as the EW-ES optimal control, with the understanding that this 
refers to the induced time consistent control at any time $t_i > t_0$. 
\end{remark}

\section{HJB Dynamic Programming  Optimization Framework}\label{algo_section}

The HJB framework uses dynamic programming, 
creating sub-problems from each time step in the problem and
moving backward in time. 
For the convenience of the reader, we will summarize the algorithm in
\citet{forsyth:2022} here.

\subsection{Deriving Auxiliary Function from $\mathit{PCEE_{t_0}}(\kappa)$}

The HJB framework begins with defining auxiliary functions based 
on the objective function (\ref{PCEE_a}) and the underlying 
stochastic processes. An  equivalent problem is then
formulated, which will then be solved to find the optimal value function. 

We begin by interchanging the $\sup_{\mathcal{P}_{0}}$ and $\sup_{W^*}$ operators. This will serve as the starting point 
for the 
HJB solution

\begin{eqnarray}
  \qquad J\left(s,b,  t_{0}^-\right)
        &= &  \sup_{W^*} \sup_{\mathcal{P}_{0}\in\mathcal{A}}
                \Biggl\{
                 E_{\mathcal{P}_{0}}^{X_0^-,t_{0}^-}
             \Biggl[ ~\sum_{i=0}^{M} q_i ~  + ~
                \kappa \biggl( W^* + \frac{1}{\alpha} \min (W_T -W^*, 0) \biggr) 
                \Biggr. \Biggr. \nonumber \\
        & & ~~~~ \Biggl. \Biggl.
               + \epsilon W_T
                  \bigg\vert X(t_0^-) = (s,b)
                     ~\Biggr] \Biggr\} ~  
               \label{pcee_inter} ~.
  \end{eqnarray}

The auxiliary function which needs to be computed in the dynamic programming framework  at each time  
$t_n$ will  have an associated strategy for any $t_n > 0$ that is equivalent with the solution of $\mathit{PCEE_{t_0}}\left(\kappa \right)$ for a fixed $W^*$. For a full discussion of pre-commitment and time-consistent ES strategies, we refer 
the reader to \citet{forsyth_2019_c}, which also includes a proof with similar steps of how the following auxiliary function is derived from (\ref{pcee_inter}).  
Including  $W^*$ in the state space gives us the 
expanded state space $\hat{X} = (s,b,W^*)$. The auxiliary function 
$V(s, b, W^*, t) \in \Omega = [0,\infty) \times (-\infty, +\infty) \times (-\infty, +\infty)  
\times [0, \infty)$ is defined as,

\begin{eqnarray}
  V(s, b, W^*, t_n^-) & = & \sup_{\mathcal{P}_{n}\in\mathcal{A}_n}
       \Biggl\{
              E_{\mathcal{P}_{n}}^{\hat{X}_n^-,t_{n}^-}
          \Biggl[
               \sum_{i=n}^{M} q_i +
          \kappa
            \biggl(
                 W^* + \frac{1}{\alpha} \min((W_T -W^*),0) 
              \biggr) \Biggr. \Biggr. \nonumber \\
    & & ~~~ \Biggl. \Biggl. + \epsilon W_T
             \bigg\vert \hat{X}(t_n^-) = (s,b, W^*)  \Biggr]
                  \Biggr\}~. \nonumber \\
                   ~  \nonumber \\
    \text{ subject to } & &
 \begin{cases}
(S_t, B_t) \text{ follow processes \eqref{jump_process_stock} and \eqref{jump_process_bond}};  
~~t \notin \mathcal{T} \\
     W_{i}^+ = S_{i}^{-} + B_{i}^{-}  - q_i \,; ~ \hat{X}_i^+ = (S_i^+ , B_i^+, W^*)  \\
  S_i^+ = p_i(\cdot) W_i^+ \,; 
~B_i^+ = (1 - p_i(\cdot) ) W_i^+ \, \\
   ( q_i(\cdot), p_i(\cdot) )  \in \mathcal{Z}(W_i^-, W_i^+ ,t_i)   \\
   i = n, \ldots, M ~;~ t_i \in \mathcal{T}  \\
              \end{cases}  ~.
            \label{expanded_1}
\end{eqnarray}

\subsection{Applying Dynamic Programming at Rebalancing Times}

The principle of dynamic programming is applied at each $t_n \in \mathcal{T}$ 
on (\ref{expanded_1}). As usual, the optimal control needs to be 
computed in  reverse time order. We split the $\sup_{\mathcal{P}_n}$ 
operator into $\sup_{q \in \mathcal{Z}_q}\sup_{p \in \mathcal{Z}_p(w^- - q, t) }$.

\begin{eqnarray}
  V(s,b,W^*, t_n^-) & = & 
                        \sup_{q \in \mathcal{Z}_q}  ~~~\sup_{p \in \mathcal{Z}_p(w^- - q, t) } 
                      \biggl\{ q+
                           \biggl[ 
                                 V(  (w^- -q) p ,  (w^- - q) (1-p) , W^*, t_n^+)
                                  \biggr]
                                              \biggr\} \nonumber \\
                    & = & \sup_{ q \in Z_q}  \biggl\{ q + 
                                \biggl[ \sup_{ p \in \mathcal{Z}_p( w^- - q, t)}
                                 V(  (w^- -q) p ,  (w^- - q) (1-p) , W^*, t_n^+)
                                  \biggr]
                                              \biggr\} \nonumber \\
                     & & w^- = s+b  ~ . \label{dynamic_a}
\end{eqnarray}

Let $\xbar{V}$ denote the upper semi-continuous envelope of $V$, which will have already been computed as the algorithm progresses 
backward through time.
The optimal allocation control $p_n(w,W^*)$ at time $t_n$
is determined from 

\begin{eqnarray}
  p_n(w, W^*) & = & \left\{
         \begin{array}{lc}
                   \displaystyle  
         \argmax_{ p^{\prime} \in [0,1] } {\xbar{V}}( w p^{\prime}, w (1 - p^{\prime}), W^*, t_n^+),
                       &  w > 0 ~;~ t_n \neq t_M\\
                        0, &  w   \leq 0 ~{\mbox{ or }} t_n = t_M
          \end{array}
               \right.  ~.
          \label{expanded_3}
\end{eqnarray}
The control  $q$ is then determined from
\begin{eqnarray}
   q_n(w,W^*) & = & \argmax_{ q^{\prime} \in \mathcal{Z}_q }
                 \biggl\{
                    q^{\prime} 
                  + {\xbar{V}}( (w -q^{\prime}) p_n( w -q^{\prime}  , W^*) ,
                       (w -q^{\prime}) ( 1 - p_n( w -q^{\prime} , W^*) ) , W^*, t_n^+)
                  \biggr\}  ~. \nonumber \\
               \label{q_opt}
\end{eqnarray}
Using these controls for $t_n$, the solution is then advanced backwards across time from $t_n^+$ to $t_n^-$ by

\begin{eqnarray}
  V( s, b, W^*,t_n^-) & = & q_n( w^-, W^*) + {\xbar{V}}( w^+ p_n(w^+,W^*), w^+(~ 1 - p_n(w^+,W^*) ~), W^*, t_n^+)
                         \nonumber \\
                      & &~~~~~~ w^- = s + b~;~ w^+ = s + b - q_n(w^-, W^*) ~. \nonumber \\
               \label{expanded_4}
\end{eqnarray}
At $t=T$, we have the terminal condition
\begin{eqnarray}
V(s, b, W^*,T^+) & = &  \kappa \biggl( 
                               W^* + \frac{\min( (s+b -W^*), 0 )}{\alpha} 
                               \biggr) ~. \label{expanded_5}
\end{eqnarray}

\subsection{Conditional Expectations between Rebalancing Times}
For $t \in(t_{n-1},t_n)$, there are no cash flows, discounting (all quantities are inflation-adjusted), 
or controls applied. Hence the tower property gives,
for $0 <h < (t_n - t_{n-1})$,
\begin{eqnarray}
   V(s,b,W^*, t) & = & E\biggl[ 
                   V( S(t+h), B(t+h), W^*, t+h) \big\vert S(t) = s, B(t) = b 
                         \biggr] ~;~ t \in( t_{n-1}, t_n -h) ~. \nonumber \\
        \label{expanded_6}
\end{eqnarray}
To find this conditional
expectation based on parametric models of the
stock and bond processes, Ito's Lemma for jump processes \citep{Cont_book} 
is first applied using Equations (\ref{jump_process_stock}) and (\ref{jump_process_bond}). 
For details of the  resulting partial integro differential equation (PIDE), 
refer to \citet{forsyth:2022} and Appendix \ref{PIDE_all}. In computational practice, 
the resulting PIDE is solved using Fourier methods discussed in \citet{forsythlabahn2017}.

\subsection{Equivalence with $\mathit{PCEE_{t_0}}(\kappa)$}

Proceeding backward in time, the auxiliary function $V(s,b, W^*,t_0^-)$ is determined
at time zero.  Problem $\mathit{PCEE_{t_0}}\left(\kappa \right)$ is then solved
using a final optimization step
\begin{eqnarray}
  J(s,b,t_0^-) = \sup_{W^\prime} V(s,b,W^{\prime},t_0^-)~.
             \label{expanded_equiv}
\end{eqnarray}
Notice that $V(s,b,W^{\prime},t_0^-)$ denotes the auxiliary function for the beginning 
of the investment period, and represents the last step (going backward)
in solving the dynamic programming formulation. To obtain this, we begin with Equation (\ref{expanded_5}) and recursively work backwards in time; then we obtain Equation (\ref{PCEE_a}) by interchanging $\sup_{W^{\prime}} \sup_{\mathcal{P}} $ in the final step. 

This formulation (\ref{expanded_1}-\ref{expanded_6}) is equivalent to problem $PCEE_{t_0}(\kappa)$. For a summary of computational details, refer to Appendix \ref{appendix_PIDE} or see \citet{forsyth:2022}.

\section{Neural Network Formulation} \label{NN_formulation_section}

As an alternative to the HJB framework, we develop a neural network  framework to solve the stochastic optimal 
control problem \eqref{PCEE_a}, which has the following characteristics:

\begin{enumerate}[label=(\roman*)]
\item The NN framework is data driven, which does not require a parametric model being specified. This avoids explicitly postulating parametric stochastic processes and the estimation of associated parameters.
In addition, this allows us to add  
auxiliary market signals/variables (although we do not exploit this idea in this work).

\item 
The NN framework avoids the  computation
of high-dimensional conditional expectations by solving for the control at all 
times directly from a single standard unconstrained optimization, 
instead of using dynamic programming (see \citet{beating_benchmark} for a discussion
of this). Since the control is low-dimensional, 
the  framework can exploit this to avoid the 
{\em curse of dimensionality} by solving for the control directly, 
instead of via value iteration such as in the HJB dynamic programming 
method \citep{beating_benchmark}. Such an approach also avoids 
backward error propagation through rebalancing times.

\item 
If the optimal control is a continuous function of time and state, the control approximated by the NN will reflect this property. If the optimal control is discontinuous, the NN approximation produces a smooth approximation. While not required by the original problem formulation in (\ref{PCEE_a}),  this continuity property likely leads to practical benefits
for an investment policy.

\item 
The NN method is further scalable in the sense that it could be easily adapted to problems with longer horizons or higher rebalancing frequency without significantly increasing the computational complexity of the problem. 
This is in contrast to existing approaches using a stacked neural network approach \citep{tsang_wong_dpl}.

\end{enumerate}

We now formally describe the proposed NN framework and demonstrate the aforementioned properties. We approximate the 
control in $\mathcal{P}$ directly by using feed-forward, fully-connected neural networks. 
Given parameters  $\bm{\theta}_p$ and $\bm{\theta}_q$,  i.e. NN weights and biases,   $\hat{p}(W(t_i), t_i, \bm{\theta}_p)$ and $\hat{q}(W(t_i), t_i, \bm{\theta}_q)$  approximate the controls $p_i$ and $q_i$ respectively,
  \begin{eqnarray}
    \hat{q}(W_i^-, t_i^-, \bm{\theta}_q) \simeq q_i(W_i^-)  ~;~ i=0, \ldots, M 
    \label{qhat_def} \nonumber \\ 
    \hat{p}(W_i^+, t_i^+, \bm{\theta}_p) \simeq p_i(W_i^+)  ~;~ i=0, \ldots, {M-1} \nonumber \\
    \hat{\mathcal{P}} = \{(\hat{q}(\cdot), \hat{p}(\cdot))\} \simeq \mathcal{P} ~~~~~~~~~~~~~~
    \nonumber 
    \label{Phat_def}
  \end{eqnarray}

The functions  $\hat{p}$ and $\hat{q}$ take time as one of the inputs, 
and therefore we can use just two NN functions to approximate  control $\mathcal{P}$ across time instead of defining a NN at each rebalancing time. In this section, we discuss how we solve  problem (\ref{PCEE_a}) using this approximation and then provide a description of the NN architecture that is used. We  discuss the precise formulation used by the NN, including activation functions that encode the stochastic constraints. 

\subsection{Neural Network Optimization for  $\mathit{PCEE_{t_0}}(\kappa)$ }

We begin by  describing the NN optimization problem based on the stochastic 
optimal control problem (\ref{PCEE_a}). We first recall  that, in the 
formulation in Section \ref{algo_section}, controls $q_i$ and $p_i$ are  functions of wealth only. 
Our goal is to choose  NN weights $\bm{\theta}_p$ and $\bm{\theta}_q$  by solving (\ref{PCEE_a}), with $\hat{q}(W_i^-,t_i^-, \bm{\theta}_q)$ and $\hat{p}(W_i^+,t_i^+, \bm{\theta}_p)$
 approximating feasible controls $(q_i, p_i) \in 
\mathcal{Z}(W_i^-, W_i^+,t_i)$ for $t_i \in \mathcal{T}$. 
For an arbitrary set of controls $\hat{\mathcal{P}}$ and wealth level $W^*$, we define the NN 
performance criteria $V_{NN}$ as
    
\begin{eqnarray}
  \qquad V_{NN}(\hat{\mathcal{P}}, W^*, s,b, t_{0}^-) & = & 
                 E_{\hat{\mathcal{P}}_0}^{X_0^-,t_{0}^-}
             \Biggl[ ~\sum_{i=0}^{M} \hat{q_i} ~  + ~
                \kappa \biggl( W^* + \frac{1}{\alpha} \min (W_T -W^*, 0) \biggr) \Biggr.  
                     \nonumber \\
    & & ~~~~~\Biggl.   + \epsilon W_T
                  \bigg\vert X(t_0^-) = (s,b)
                     ~\Biggr]  ~.  \nonumber \\
& & \text{ subject to } 
               \begin{cases}
(S_t, B_t) \text{ follow processes \eqref{jump_process_stock} and \eqref{jump_process_bond}};  
 ~~t \notin \mathcal{T} \\
      W_{i}^+ = S_{i}^{-} + B_{i}^{-}  - q_i \,; ~ X_i^+ = (S_i^+ , B_i^+)  \\
   S_i^+ = \hat{p}_i(\cdot) W_i^+ \,; 
 ~B_i^+ = (1 - \hat{p}_i(\cdot) ) W_i^+ \, \\
    ( \hat{q}_i(\cdot) , \hat{p}_i(\cdot) )  \in \mathcal{Z}(W_i^-, W_i^+,t_i)  \\
    i = 0, \ldots, M ~;~ t_i \in \mathcal{T}  \\
               \end{cases}  ~.
               \label{pcee_nn_a}
  \end{eqnarray}
The optimal value function $J_{NN}$ (at $t_0^-$) is then given by
\begin{eqnarray}
    J_{NN}( s, b, t_0^-) = \sup_{W^*}\sup_{\hat{\mathcal{P}} \in \mathcal{A} } 
                 ~V_{NN}(\hat{\mathcal{P}}, W^*, s,b, t_{0}^-)
  ~.
    \label{PIDE_nn_equivalence}
\end{eqnarray}
Next we  describe the structure of the neural networks and feasibility encoding.

\subsection{Neural Network Framework} \label{NN_framework}

 Consider two fully-connected feed-forward NNs, with $\hat{p}$ and $\hat{q}$ determined by parameter vectors $\bm{\theta}_p \in \mathbb{R}^{\nu_p}$ and $\bm{\theta}_q \in \mathbb{R}^{\nu_q}$ (representing NN weights and biases), respectively. 
The two NNs can differ in the choice of activation functions and in the number of hidden layers and nodes per layer. 
 Each NN takes input of the same form $(W(t_i),t_i) $, but the withdrawal NN $\hat{q}$ takes the 
state variable observed  before withdrawal, 
$(W(t_i^-), t_i)$, and the allocation NN $\hat{p}$ takes the 
state variable observed after withdrawal, $(W(t_i^+), t_i)$.

In order for the NN to generate a feasible control as specified in \eqref{PCEE_nn_estimate}, 
we use a modified sigmoid activation function to scale the output from 
the withdrawal NN $\hat{q}$ according to the $PCEE_{t_0}(\kappa)$ problem's 
constraints on the withdrawal amount $q_i$, as given in Equation (\ref{Z_q_def}). 
This ultimately allows us to perform unconstrained optimization on the NN training parameters. 

Specifically, assuming  $x \in [0,1]$,  the function $f(x) := a + (b-a)x$ scales
the output to be in the range $[a,b]$. We 
restrict withdrawal to $\hat{q}$ in $[q_{\min}, q_{\max}]$. We note
 that this withdrawal range $q_{\max}- q_{\min}$ depends 
on wealth $W^-$,  see from (\ref{Z_q_def}).  Define the range of permitted withdrawal as follows,

$$ \text{range} =
\begin{cases}
              q_{\max} - q_{\min} ~;~  \text{if }  W_i^- > q_{\max}\\
              W^- - q_{\min} ~;~~\text{if }  q_{\min} < W_i^- < q_{\max} \\
              0 ~;~ ~~~~~~~~~~~~~~\text{if }  W_i^- < q_{\min}
\end{cases} ~.
$$
More concisely, we have the following mathematical expression:
$$
\text{range} = \max\left((\min(q_{\max},W^-) - q_{\min}), 0 \right).
$$
Let $z \in \mathbb{R}$ be the NN output before the final output layer of $\hat{q}$. Note that $z$ 
depends on  the input features,  state and time,
before being transformed by the activation function. 
We then have the following expression for the  withdrawal,
\begin{eqnarray*}
\hat{q}(W^-, t, \bm{\theta_q})& = &q_{\min}  + \text{range}  \cdot \biggl( \frac{1}{1+e^{-z}} \biggr)\\
 &= &  q_{\min} + 
           \max\left((\min(q_{\max},W^-) - q_{\min}), 0\right) \biggl( \frac{1}{1+e^{-z}} \biggr)
~.
\end{eqnarray*}
Note that  the sigmoid function $\tfrac{1}{1 + e^{-z}}$  is a mapping from $\mathbb{R} \to [0,1]$.

Similarly, we use a softmax activation function on the  NN output of the $\hat{p}$, in order to impose
 no-shorting and no-leverage constraints.

With these output activation functions, it can be easily verified that    $ \left( \hat{q}_i(\cdot) , \hat{p}_i(\cdot) \right)  \in \mathcal{Z}\left( W_i^-, W_i^+ , t_i \right) $ always.
Using the defined NN, this transforms the problem (\ref{PIDE_nn_equivalence})  
of finding an optimal $\hat{\mathcal{P}}$ 
into  the optimization problem:
\begin{eqnarray}
    \hat{J}_{NN} (s, b, t_0^-) &=& \sup_{W^* \in \mathbb{R}}\sup_{\bm{\theta_q} 
                \in \mathbb{R}^{\nu_q}}\sup_{\bm{\theta_p} \in \mathbb{R}^{\nu_p}} 
           ~\hat{V}_{NN}(\bm{\theta}_q,\bm{\theta}_p, W^*, s,b, t_{0}^-)  \nonumber \\
    &=& \sup_{(W^*,\bm{\theta_q}, \bm{\theta_p}) \in 
             \mathbb{R}^{\nu_q + \nu_p + 1}} ~\hat{V}_{NN}(\bm{\theta}_q,\bm{\theta}_p, W^*, s,b, t_{0}^-)
    \label{pcee_with_inf2}   ~.
\end{eqnarray}
It is worth noting here that, while the original control $\mathcal{P}$ is 
constrained in (\ref{P_constraint_set}), the formulation (\ref{pcee_with_inf2}) 
is an unconstrained  optimization over $\bm{\theta}_q$, $\bm{\theta}_p$, and $W^*$. 
Hence we can solve problem (\ref{pcee_with_inf2}) directly using a standard 
gradient descent technique. In the numerical experiments detailed in 
Sections \ref{results_section} and \ref{testing_section}, we use Adam 
stochastic gradient descent \citep{kingma2017adam} to 
determine the optimal points $\bm{\theta_q}^*$, $\bm{\theta_p}^*$, and $W^*$.

Note that the output of NN $\hat{q}$ yields the amount to 
withdraw, while the output of NN $\hat{p}$ produces asset allocation weights.

Figure \ref{NN_diagram} presents the proposed NN. We emphasize the following key aspects of this NN structure.

\begin{enumerate}[label=(\roman*)]
\item Time is an \emph{input} to both NNs in the framework. Therefore, 
the parameter vectors $\bm{\theta}_q$ and $\bm{\theta}_p$ are constant and do not vary with time. 

\item At each rebalancing time, the wealth observation  before withdrawal is used 
to construct the feature vector for $\hat{q}$.  The resulting withdrawal is then used 
to calculate wealth after withdrawal, which is an input feature for $\hat{p}$.

\item Standard sigmoid activation functions are used at each \emph{hidden layer} output.  

\item The activation function for the withdrawal output is 
different from the activation function for allocation. Control  $\hat{q}$ uses a modified sigmoid function, which is chosen to transform its output according to (\ref{Z_q_def}). Control  $\hat{p}$ uses a softmax activation which ensures that its output gives only positive weights 
for each portfolio asset that sum to one, as specified in (\ref{Z_p_def}). By constraining the NN output this way through proposed activation functions, we can use unconstrained optimization  to train the NN. 
\end{enumerate}

\begin{figure}[ht!]
\centering
\includegraphics[width=0.99\linewidth]{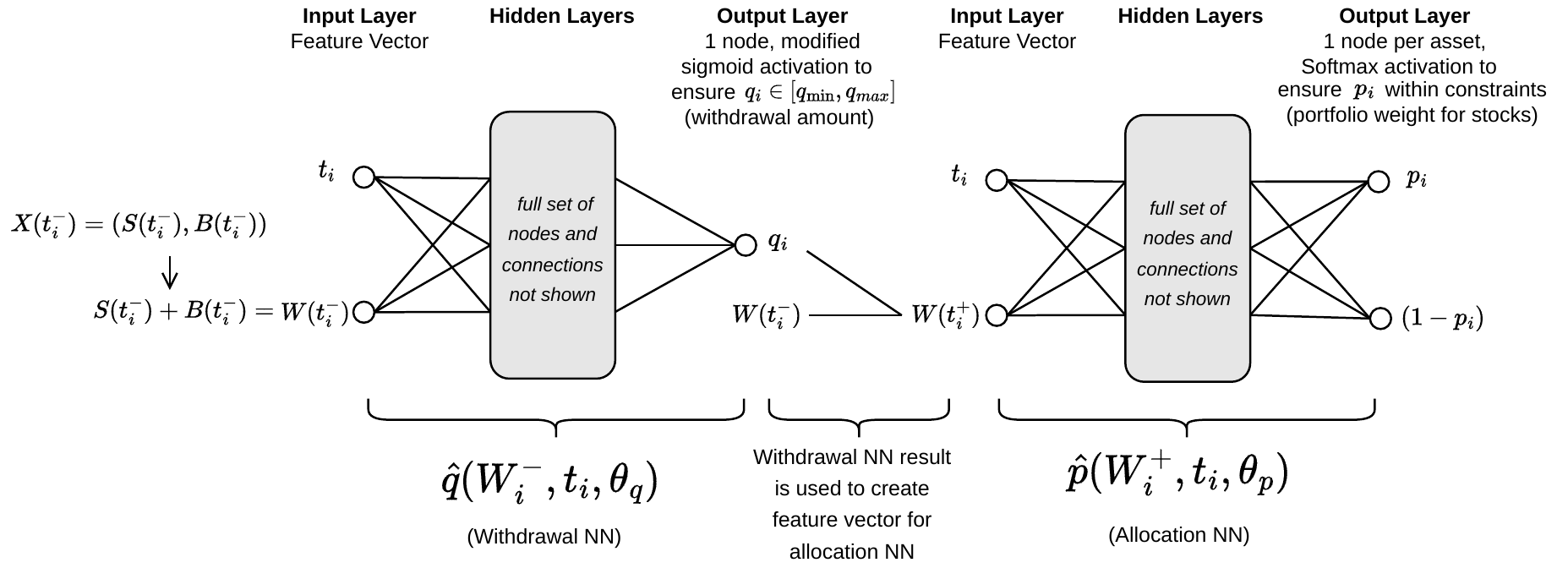}
\caption{Illustration of the NN framework as per Section \ref{NN_framework}. Additional technical details can be found in Appendix \ref{appendix_nn}.}
\label{NN_diagram}
\end{figure}

\subsection{NN Estimate of the Optimal Control}
\label{NN_estimate_of_optimal_control}

Now we describe the training  optimization problem for the  proposed data driven NN framework,  which is 
agnostic to the underlying data generation process. We assume that
a set of asset return trajectories are available,
which are used  to approximate the expectation in (\ref{pcee_nn_a}) for any given control.  
For NN training, we approximate 
the expectation in (\ref{pcee_nn_a}) based on  a finite number of samples as follows:

\begin{multline}
  \hat{V}_{NN}(\bm{\theta}_q,\bm{\theta}_p, W^*, s,b, t_{0}^-)  =  \\
                 \tfrac{1}{N} \sum_{j=1}^{N}
             \Biggl[ ~\sum_{i=0}^{M} \hat{q}( (W_i)^{j}, t_i; \bm{\theta}_q) ~  + ~
                \kappa \biggl( W^* + \frac{1}{\alpha} \min ( (W_T)^j -W^*, 0) \biggr) + \epsilon (W_T)^j \bigg\vert X(t_0^-) = (s,b)
                     ~\Biggr]  ~  
               \\
               \text{ subject to } 
               \begin{cases}
((S_t)^j, (B_t)^j) \text{ drawn from the $j^{th}$ sample of returns};  
 ~~t \notin \mathcal{T} \\
      (W_{i}^+)^j = (S_{i}^{-})^j + (B_{i}^{-})^j  
                - \hat{q} \left( (W_{t_i}^-)^j,t_i, \bm{\theta_q} \right) \,; ~ (X_i^+)^j = (S_i^+ , B_i^+)^j  \\
   (S_i^+)^j = \hat{p} \left( (W_{i}^+)^j, t_i, \bm{\theta_p} \right) ~(W_i^+)^j \,; 
 ~(B_i^+)^j = (1 - \hat{p} \left( (W_{i}^+)^j, t_i, \bm{\theta_p}) \right) ~(W_i^+)^j \, \\
    \left( \hat{q}_i(\cdot) , \hat{p}_i(\cdot) \right)  \in \mathcal{Z}\left( (W_i^-)^j, (W_i^+)^j , t_i \right)  \\
    i = 0, \ldots, M ~;~ t_i \in \mathcal{T}  \\
               \end{cases}  ~,
\label{PCEE_nn_estimate}
\end{multline}

\noindent where the superscript $j$ represents the $j^{\text{th}}$ path of joint asset returns and $N$ is the total number of sampled paths. 
For subsequent benchmark comparison, we generate price paths using processes \eqref{jump_process_stock} and \eqref{jump_process_bond}.
We are, however, agnostic as to the method used to generate these paths.
We assume that the random sample paths are independent, but that
correlations can exist  between returns of different assets.
In addition,  correlation between the returns of different time periods can also be represented,  
e.g., block bootstrap resampling is designed to capture autocorrelation in the time series data.

The optimal parameters obtained by training the neural network are used to generate the control functions $\hat{q}^* (\cdot):= \hat{q}(\cdot ; \bm{\theta_q^*})$ and $\hat{p}^*(\cdot) := \hat{p}(\cdot ; \bm{\theta_p^*})$, respectively. With these functions, we can evaluate the performance of the generated control on testing data sets that are out-of-sample or out-of-distribution. We present the detailed results of such tests in Section \ref{testing_section}.     

\section{Data} \label{data_description_section}

For the computational study in this paper, we use data from the Center for Research in Security Prices (CRSP) 
on a monthly basis over the 1926:1-2019:12 period.\footnote{More specifically, results presented here
were calculated based on data from Historical Indexes, \copyright
2020 Center for Research in Security Prices (CRSP), The University of
Chicago Booth School of Business. Wharton Research Data Services was
used in preparing this article. This service and the data available
thereon constitute valuable intellectual property and trade secrets
of WRDS and/or its third-party suppliers.} The specific indices used are the CRSP 10-year U.S. Treasury 
index for the bond asset\footnote{The 10-year Treasury index was calculated using monthly returns from CRSP dating back to 1941. 
The data for 1926-1941 were interpolated from annual returns in \citet{Homer_rates}. The bond index is constructed
by (i) purchasing a 10-year Treasury at the start of each month, (ii) collecting interest during the month and
(iii) selling the Treasury at the end of the month.} 
and the CRSP value-weighted total return index for the stock asset\footnote{The stock index 
includes all distributions for all domestic stocks trading on major U.S. exchanges.}. 
All of these various indexes
are in nominal terms, so we adjust them for inflation by using the U.S.\
CPI index, also supplied by CRSP. We use real indexes since investors
funding retirement spending should be focused on real (not nominal)
wealth goals.

We use the above market data in two different ways in subsequent investigations: 

\begin{enumerate}  [label = (\roman*)]
    \item \emph{Stochastic model calibration:} Any data set referred to in this paper as {\em synthetic data} 
is generated by parametric stochastic models (SDEs) (as described in Section \ref{Market Model Section}), whose parameters are calibrated to the CRSP data by using the threshold technique \citep{mancini2009,contmancini2011,Dang2015a}. The data is inflation-adjusted so that all parameters reflect real returns. Table \ref{fit_params} shows the results of calibrating the models to the historical data. 
The correlation $\rho_{sb}$ is computed by removing any returns which occur at 
times corresponding to jumps in either series.
See \citet{Dang2015a} for details of the technique
for detecting jumps.

    \item  
    \emph{Bootstrap resampling:} Any data set referred to in this paper as {\em historical data} 
is generated by using the stationary block bootstrap method \citep{politis1994,politis2004,politis2009,dichtl2016testing} to resample the historical CRSP data set. 
This method involves repeatedly drawing randomly sampled blocks 
of random size, with replacement, from the original data set. The
block size follows a geometric distribution with  a specified expected block size.
To preserve correlation between asset returns, we use a paired sampling approach 
to simultaneously draw returns from both time series. This, in effect, shuffles
the original data and can be repeated to obtain however many resampled paths one desires. Since 
the order of returns in the sequence is unchanged within the sampled block, 
this method accounts for some possible serial correlation in 
market data. Detailed pseudo-code for this method of block 
bootstrap resampling is given in \citet{Forsyth_Vetzal_2019a}. 

We note that block resampling is commonly used by  practitioners and
academics (see for example 
\citet{anarkulova2022stocks, dichtl2016testing, scott2017wealth, simonian2022sharpe, cogneau2013block}). 
Block bootstrap resampling will be used to carry out 
robustness checks in Section \ref{testing_section}.
Note that for any realistic number of samples and expected
block size, the probability of repeating a resampled path
is negligible \citep{ni2022optimal}.
    
 One important parameter for the block resampling method is the expected block size.     
\citet{forsyth:2022} determines that a reasonable expected block size for paired resampling is about 
three months. The algorithm presented in \citet{politis2009} is used to determine the optimal expected block size for the bond and stock returns separately; see Table \ref{auto_blocksize}.
Subsequently, we will also test the sensitivity of the
results to a range of block sizes from 1 to 12 months in numerical experiments. 

\end{enumerate}

To train the neural networks, we require that the number of sampled paths, $N$, be sufficiently large to fully represent the underlying market dynamics. 
Subsequently, we  first generate training data through Monte Carlo simulations of the parametric models described in \eqref{jump_process_stock} and \eqref{jump_process_bond}. 
We emphasize however that in the proposed data driven NN framework,  
we only require  return trajectories of the underlying assets.  
In later sections,  
we present results from NNs trained on  non-parametrically generated data, e.g.
resampled historical data. 
We also demonstrate the NN framework's robustness on test data. 

\section{Computational Results} \label{results_section}
We now present and compare performance of the  optimal
control from the HJB PDE and NN method respectively on synthetic data, with 
investment specifications given  in Table \ref{base_case_1}. 
Each strategy's performance is measured w.r.t. to the objective function  in (\ref{PCEE_a}), which is   a weighted reward (EW) and risk (ES) measure.
To trace out an efficient frontier in the (EW,ES) plane,
we vary $\kappa$ (the curve represents the (EW,ES)  performance  on a set of optimal Pareto points).   

We first present strategies computed  from 
the HJB framework described in Section 3.  We verify that the numerical
solutions are sufficiently accurate, so that this solution can be regarded
as ground truth.
We then present results computed using the NN framework of Section \ref{NN_formulation_section}, 
and demonstrate the accuracy of the NN results by comparing 
to the  ground truth  computed from the HJB equation. 
We carry out further analysis by selecting an \emph{interesting} point on the (EW,ES) efficient frontier, in particular $\kappa = 1.0$, to study in greater detail. The point $\kappa = 1.0$ is at the \emph{knee} of the efficient frontier, which makes it desirable in terms of risk-reward tradeoff (picking the exact $\kappa$ will be a matter of investor preference, however). This notion of the knee point is loosely based on the concept of a \emph{compromise solution} of multi-objective optimization problems, which selects the point on the efficient frontier with the 
minimum distance to an unattainable ideal point \citep{pareto_optimality}. 
For this knee point of $\kappa = 1.0$, we analyze the controls and 
wealth outcomes under both frameworks. We also discuss some key differences 
between the HJB and NN frameworks' results and their implications.

\begin{table}[hbt!]
\begin{center}
\begin{tabular}{lc} \toprule
Investment horizon $T$ (years) & 30  \\
Equity market index & CRSP Cap-weighted index (real) \\
Bond index & 10-year Treasury (US) (real) \\
Initial portfolio value $W_0$  & 1000 \\
Cash withdrawal times & $t=0,1,\ldots, 30$\\
Withdrawal range   & $[35, 60]$\\
Equity fraction range & $[0,1]$\\
Borrowing spread $\mu_c^{b}$ & 0.0 \\
Rebalancing interval (years) & 1  \\
Market parameters & See Appendix~\ref{appendix_market_model} \\ \bottomrule
\end{tabular}
\caption{Problem setup and input data.  Monetary units: thousands of dollars.
\label{base_case_1}}
\end{center}
\end{table}

\subsection{Strategies Computed  from HJB Equation} \label{PIDE_results_section}

 We carry out a convergence test for the HJB framework by tracing the 
efficient frontier (i.e. varying the scalarization parameter $\kappa$) 
for solutions of varying refinement levels (i.e.
number of grid points in the $(s,b)$ directions). Figure \ref{frontier_convergence_plot} shows these efficient frontiers. As the efficient frontiers from  various grid sizes all practically overlap each other,  this demonstrates  convergence of solutions computed from solving HJB equations. 
Table \ref{conservative_accuracy} shows  a convergence test for a single point on the frontier. 
The convergence is roughly first-order (for the value function). This convergence test justifies the use of the 
HJB framework results as a ground-truth.

\begin{figure}[htb!]

\centering
\includegraphics[width=0.4\linewidth]{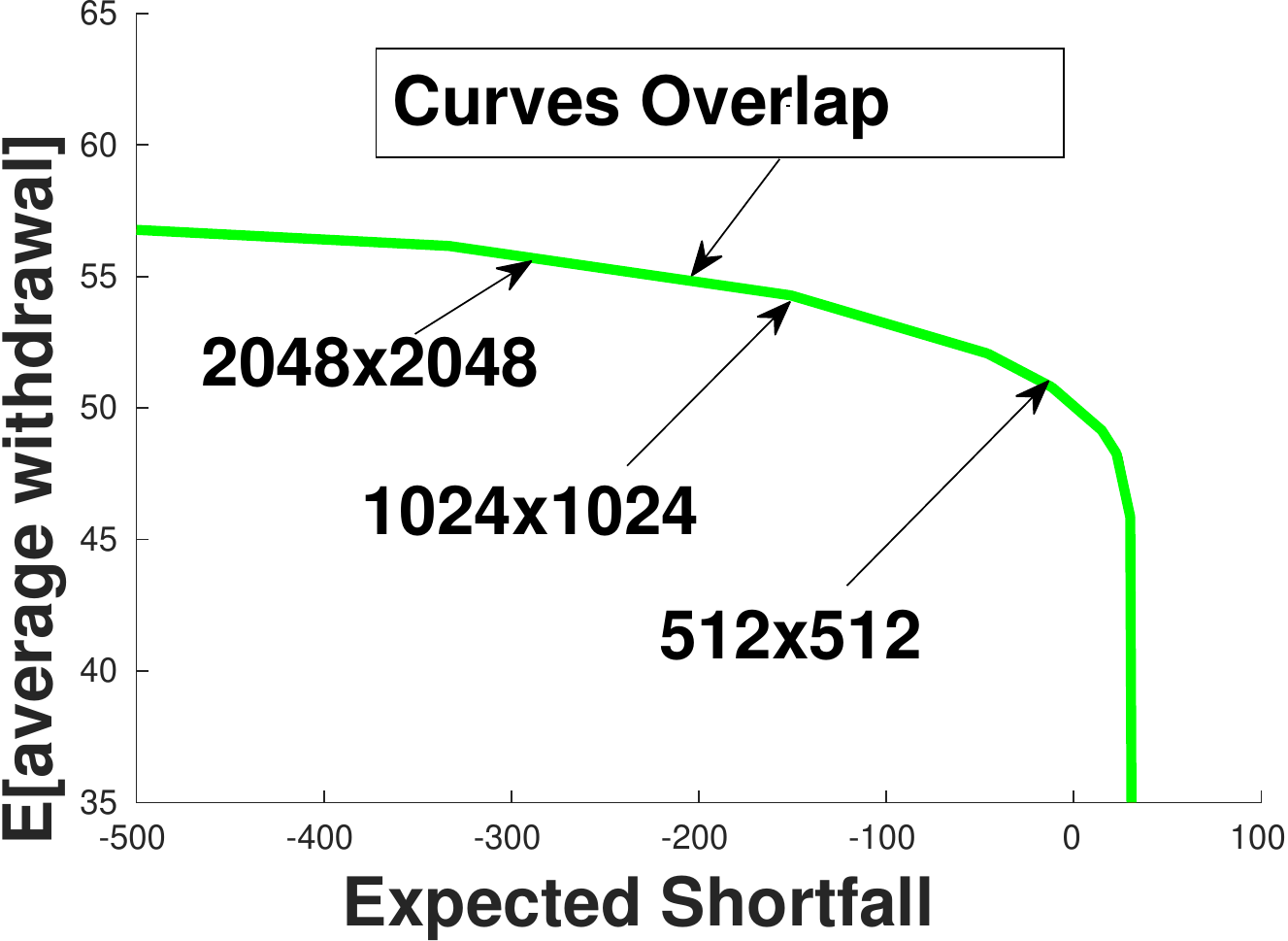}

\caption{EW-ES frontier, computed from  problem (\ref{PCEE_a}). Note: Scenario in Table \ref{base_case_1}. Comparison of HJB solution performance with varying grid sizes. HJB solution performance computed on $2.56 \times 10^6$ observations of synthetic data. Parameters for synthetic data based on cap-weighted real CRSP, real 10 year treasuries ({see Table \ref{fit_params}}).
$q_{min} = 35, q_{\max} = 60$.
$\epsilon = 10^{-6}$.
Units: thousands of dollars.}
\label{frontier_convergence_plot}
\end{figure}

\begin{remark}[Effect of Stabilization Term $\epsilon W_T$] \label{stabilization_remark}
Recall  the stabilization term, $\epsilon W_T$, introduced in (\ref{PCEE_a}). 
We now provide motivation for its inclusion, and observe its effect on the control $\hat{\mathcal{P}}$. When 
$W_t \gg W^*$ and $t \to T$, the control will only weakly affect the objective function. This is because, in 
this situation, $Pr[ W_T < W^*] \simeq 0$ and thus the allocation
control will have little effect on the ES term in the objective 
(recall that $W^*$ is held constant for the induced time consistent
strategy, see Appendix \ref{time_consistent_appendix}). 
In addition, the withdrawal is capped at $q_{\max}$ for very high values of $W_t$,
so the withdrawal control does not depend on $W_t$
in this case either.
The stabilization term can be used to alleviate ill-posedness of the problem in this region.
\end{remark}

In Figure \ref{heat_map_fig_neg_eps}, we present the heat map of the allocation control  computed from the HJB framework. Subplot (a) presents allocation control heat map 
for  a small positive stabilization parameter $\epsilon = 10^{-6}$, while Subplot (b) presents  allocation control heat map with  $\epsilon = -10^{-6}$. In the ill-posed region (the top right region of the heat maps), the presence of $\epsilon W_T$, with $\epsilon = 10^{-6}$,  forces the control to invest 100\% in stocks to generate 
high terminal wealth. Conversely, changing the stabilization parameter to be negative ($\epsilon = -10^{-6}$) forces the control to invest completely in bonds. 

We observe that the control behaves differently only at high level of wealth as $t \to T$  in both cases.   The 5th and the 50th percentiles of control on the synthetic data set behave similarly in both the positive and negative $\epsilon$ cases. The 95th percentile curve tends towards higher wealth during later phases of the investment period when the $\epsilon$ is positive (Figure \ref{PIDE_heat_stocks_35_60_poseps}), whereas the curve tends downward when $\epsilon$ is negative (Figure \ref{PIDE_heat_withdraw_35_60_neg_eps}). 
When the magnitude of $\epsilon$ is sufficiently small,  the 
inclusion of $\epsilon W_T$ in the objective function does not change summary statistics (to four decimal places when $ |\epsilon|= 10^{-6}$). While the choice of negative or positive  
$\epsilon$ with small magnitude can lead to different allocation control 
scenarios at high wealth level near the end of time horizon, the choice makes little difference from the perspective of the problem $PCEE_{t_0}(\kappa)$. 
If the investor reaches very high wealth near $T$, the choice between 100\% stocks and 100\% bonds 
does not matter as the investor always ends with $W_T \gg W^*$.
Our experiments show that the control $q$ is unaffected  when the magnitude of  $\epsilon$ is small and 
continues to call for maximum withdrawals at high levels 
of wealth as $t \to T$, just as described in Remark \ref{stabilization_remark}.

Comparing the optimal withdrawal strategy determined
by solving  stochastic optimal control problem \eqref{PCEE_a} 
with a fixed withdrawal strategy 
(both strategies with dynamic asset allocation), \citet{forsyth:2022} finds that the stochastic optimal strategy  \eqref{PCEE_nn_estimate} is much more efficient in withdrawing cash over the investment horizon. 
Accepting a very small amount of additional risk, the retiree can 
dramatically increase total withdrawals. For a more detailed
discussion  of the optimal control,
we refer the reader to \citet{forsyth:2022}.

\begin{figure}[htb!]
\centering
\textbf{} \medskip \par
\centerline{%
\begin{subfigure}[t]{.40\linewidth}
\centering
\includegraphics[width=\linewidth]{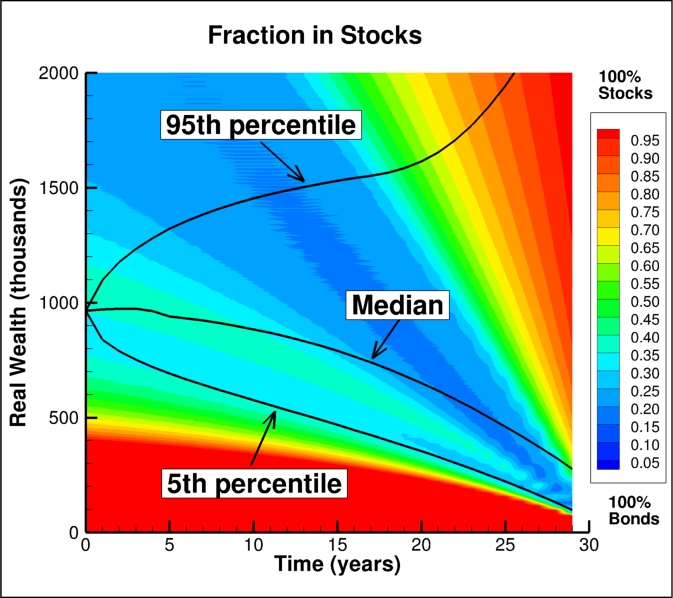}
\caption{Fraction in stocks, HJB Control, $\epsilon = 10^{-6}$}
\label{PIDE_heat_stocks_35_60_poseps}
\end{subfigure}
\hspace{.03\linewidth}
\begin{subfigure}[t]{.40\linewidth}
\centering
\includegraphics[width=\linewidth]{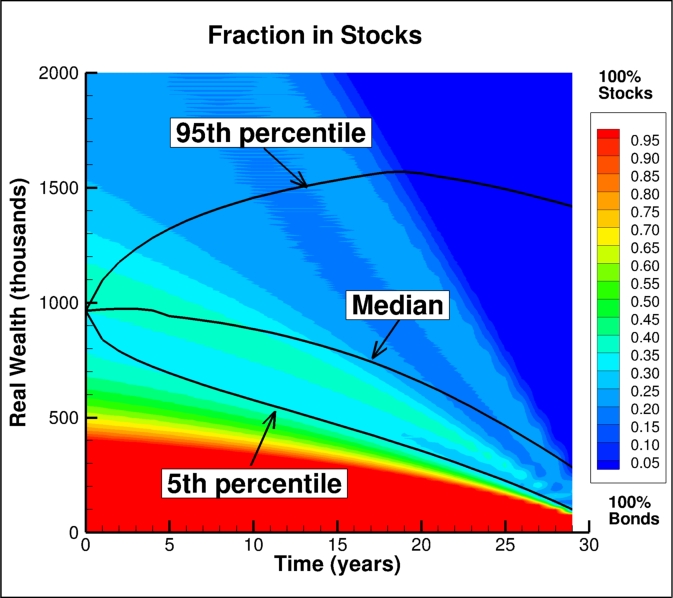}
\caption{Fraction in stocks, HJB Control, $\epsilon = -10^{-6}$}
\label{PIDE_heat_withdraw_35_60_neg_eps}
\end{subfigure}
}
\caption{
Effect of $\epsilon$: fraction in stocks computed from the problem (\ref{PCEE_a}).  Note: investment setup is as in Table \ref{base_case_1}. HJB solution performance computed on $2.56 \times 10^6$ observations of synthetic data. Parameters for synthetic data based on cap-weighted real CRSP, real 10 year treasuries (see Table \ref{fit_params}).
$q_{min} = 35, q_{\max} = 60$, $\kappa = 1.0$. $W^* = 58.0$ for PIDE results.
(a)$~\epsilon = 10^{-6}$. (b)$~\epsilon = -10^{-6}$.
Units: thousands of dollars.
\label{heat_map_fig_neg_eps}}
\end{figure}

\subsection{Accuracy of Strategy Computed  from NN framework} \label{NN_results}

We compute  the NN control following the framework discussed in Section \ref{NN_formulation_section}. 
We compare the efficient frontiers obtained from the HJB equation solution
and the NN solution.
From  Figure \ref{EF_PIDE_nn}, the NN control efficient frontier is almost 
indistinguishable from the HJB control efficient frontier. 
Detailed summary statistics for each computed point on the frontier can 
be found in Appendix \ref{nn_ef_details}, and a comparison of 
objective function values, for the NN and HJB control at each 
frontier point, can be found in Appendix \ref{appendix_obj_comparison}. 
For most points on the frontier, the difference in objective function values, from NN and HJB, is less than $0.1\%$. 
This demonstrates that the accuracy of the NN framework 
approximation of the ground-truth solution is more than adequate,
considering that the difference between the NN solution and
the PDE solution is about the same as the 
estimated PDE error (see Table \ref{conservative_accuracy}).

\begin{figure}[htb!]
\centering
\includegraphics[width=0.4\linewidth]{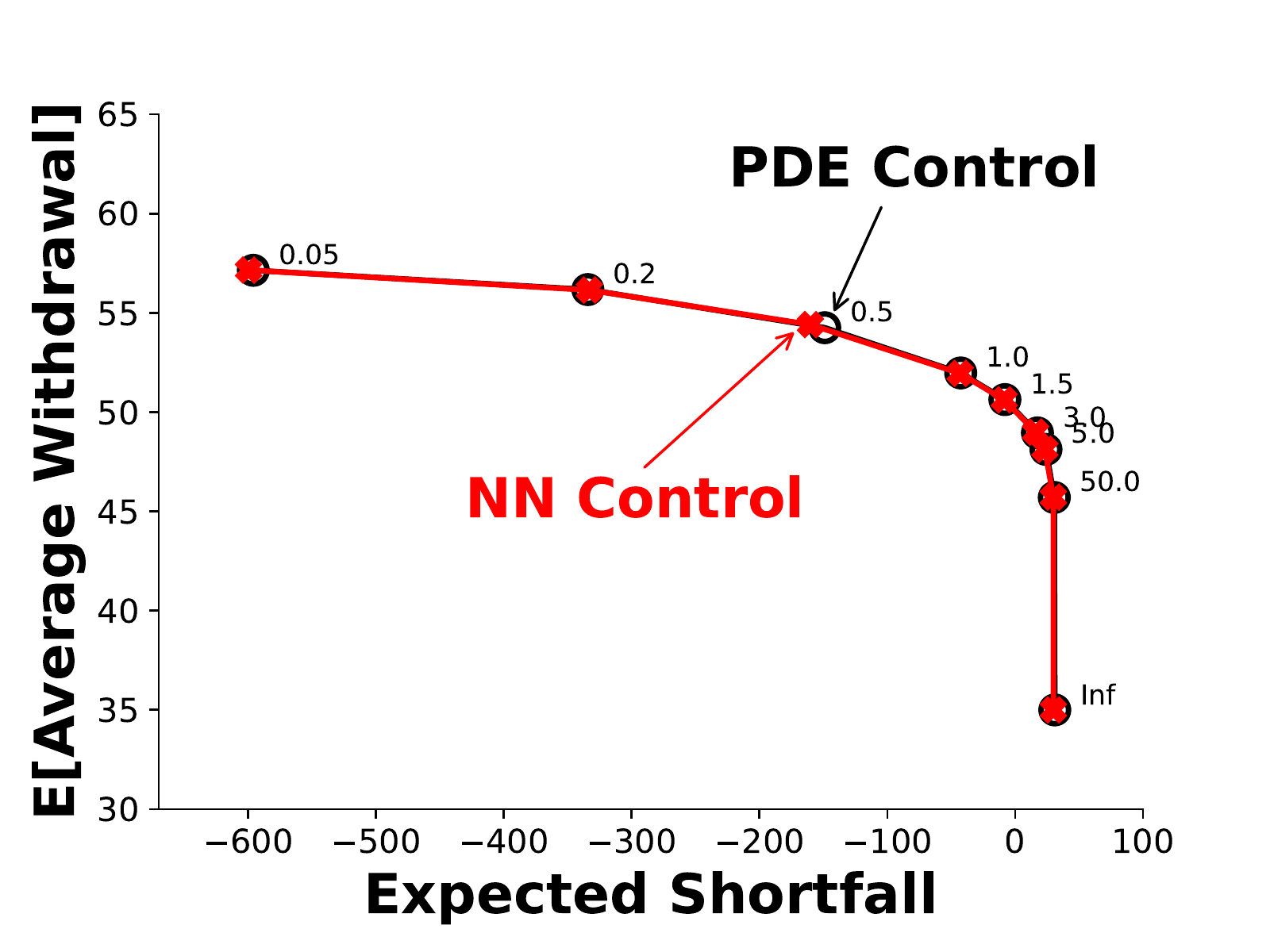}
\caption{Comparison of EW-ES frontier for the Neural Network (NN) and Hamilton-Jacobi-Bellman (HJB)  Partial
Differential Equation (PDE) methods, computed from the problem (\ref{PCEE_a}). Note: investment setup in Table \ref{base_case_1}. HJB solution performance computed on $2.56 \times 10^6$ observations of synthetic data. Parameters for synthetic data based on cap-weighted real CRSP, real 10 year treasuries (see Table \ref{fit_params}). Control computed from the NN model, trained on $2.56 \times 10^6$ observations of synthetic data.
$q_{min} = 35, q_{\max} = 60$.
$\epsilon = 10^{-6}$.
Units: thousands of dollars. Labels on nodes indicate $\kappa$ parameter.}
\label{EF_PIDE_nn}
\end{figure}

We now further analyze the control $\hat{\mathcal{P}}$ produced 
by the NN framework for $\kappa = 1$.  Comparing Figure \ref{nn_heat_withdraw_35_60} with Figure \ref{PIDE_heat_withdraw_35_60}, we observe that the withdrawal control $\hat{q}$ produced by the NN  is practically identical 
to the withdrawal control
produced by the HJB framework. However, 
there are differences in the allocation control heat maps. The NN heat map for allocation control $p$ (Figure \ref{nn_heat_stocks_35_60}) appears most similar to that of the HJB allocation heat map for negative $\epsilon$ (Figure \ref{PIDE_heat_withdraw_35_60_neg_eps}), but it is clear that the NN allocation heat map differs significantly from the HJB heat map for positive $\epsilon$ (Figure \ref{PIDE_heat_stocks_35_60_poseps}) at high level of wealth as $t \to T$. The NN allocation control behaves differently from the HJB controls in this region, choosing a mix of stocks and bonds instead of choosing a 100\% allocation in a single asset. Noting this difference is only at higher level of wealth near $T$, we see that the 5th percentile and the median wealth curves are indistinguishable. The NN control's 95th percentile curve, however, is different and indeed the curve is in between the  95th percentile curves from the negative and positive versions of the HJB-generated control.

Drawing from this, we attribute the NN framework's inability to fully replicate 
the HJB control to the ill-posedness of the optimal control problem 
in the (top-right) region of high wealth levels near $T$. 
The small value of $\epsilon$ means that the stabilization term 
contributes a very small fraction of the objective function value 
and thus has a very small gradient, relative to the first two 
terms in the objective function.  Since we use stochastic gradient 
descent for optimization, we see a very small impact of $\epsilon$. 
Moreover, the data for high levels of wealth as $t \to T$ is very sparse 
and so the effect of the small gradient is further reduced. As a result, the 
NN appears to smoothly extrapolate in this region and therefore avoids investment into a single asset. Recall that in Section \ref{PIDE_results_section}, we stated that the choice in the signs of $\epsilon$,  with small $\epsilon$,  in the stabilization term is somewhat arbitrary and does not affect summary 
statistics. 
Therefore, we see that the controls 
produced by the two methods only differ in irrelevant aspects, 
at least based on the EW and ES reward-risk consideration.

\begin{figure}[htb!]
\centering
\textbf{NN Control Results} \medskip \par
\centerline{%
\begin{subfigure}[t]{.40\linewidth}
\centering
\includegraphics[width=\linewidth]{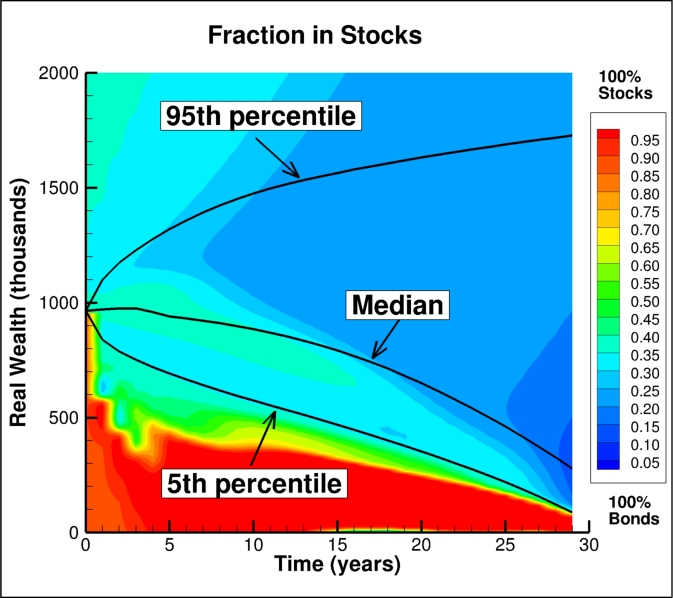}
\caption{Fraction in stocks, NN Control}
\label{nn_heat_stocks_35_60}
\end{subfigure}
\hspace{.03\linewidth}
\begin{subfigure}[t]{.40\linewidth}
\centering
\includegraphics[width=\linewidth]{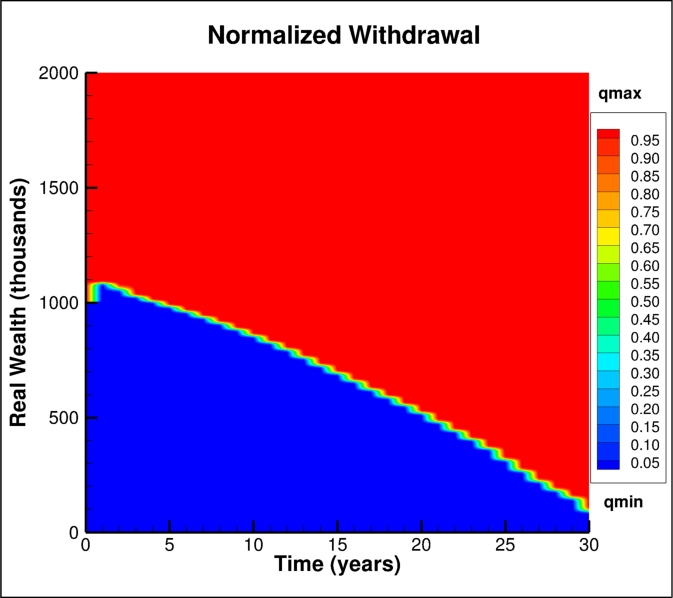}
\caption{Withdrawals, NN Control}
\label{nn_heat_withdraw_35_60}
\end{subfigure}
}
\textbf{ }\par\medskip
\centering 
\textbf{HJB Control Results}  \medskip \par
\centerline{%
\begin{subfigure}[t]{.40\linewidth}
\centering
\includegraphics[width=\linewidth]{Figures/heat_map_corrected_p_control_percentiles_peter.jpeg}
\caption{Fraction in stocks, HJB Control}
\label{PIDE_heat_stocks_35_60}
\end{subfigure}
\hspace{.03\linewidth}
\begin{subfigure}[t]{.40\linewidth}
\centering
\includegraphics[width=\linewidth]{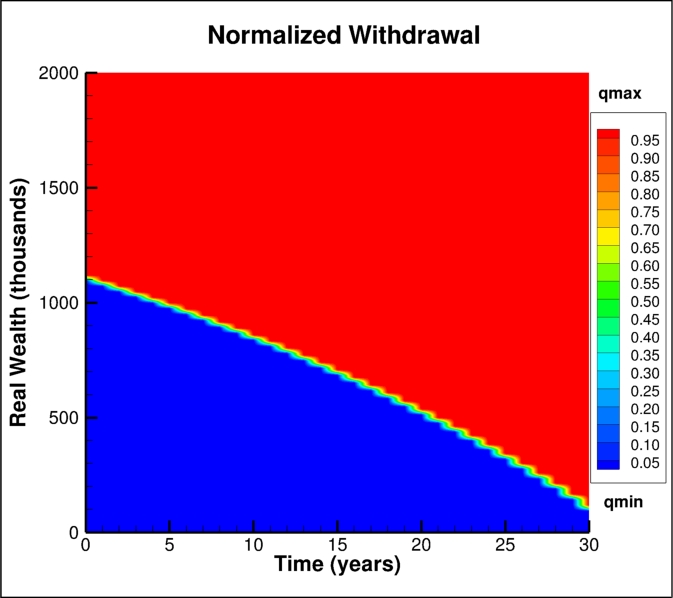}
\caption{Withdrawals, HJB Control}
\label{PIDE_heat_withdraw_35_60}
\end{subfigure}
}
\caption{
Heat map of controls: fraction in stocks and
withdrawals, computed from the problem (\ref{PCEE_a}).
 Note: problem setup described in Table \ref{base_case_1}. HJB solution performance computed on $2.56 \times 10^6$ observations of synthetic data. Parameters for synthetic data based on cap-weighted real CRSP, real 10 year treasuries (see Table \ref{fit_params}).
NN model trained on $2.56 \times 10^6$ observations of synthetic data.
$q_{min} = 35, q_{\max} = 60$, $\kappa = 1.0$. $W^* = 59.1$ for NN results. $W^* = 58.0$ for the HJB results.
$\epsilon = 10^{-6}$.
Normalized withdrawal $(q - q_{\min})/(q_{\max} - q_{\min})$.
Units: thousands of dollars.
\label{heat_map_fig}}
\end{figure}

\begin{figure}[h!]
\centering
\textbf{NN Control Results}
\centerline{%
\begin{subfigure}[t]{.33\linewidth}
\centering
\includegraphics[width=\linewidth]{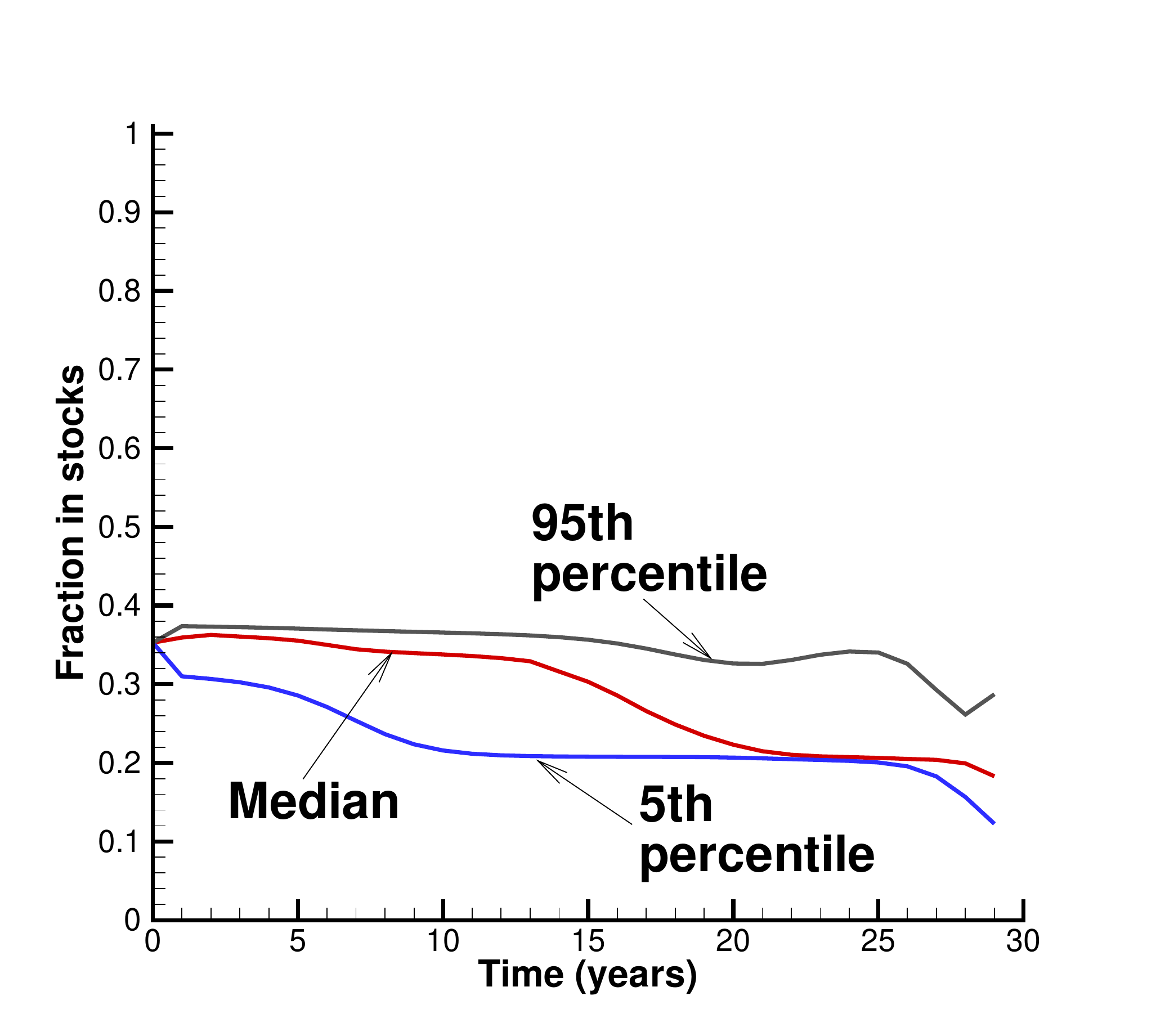}
\caption{Percentiles fraction in stocks, NN Control, $\epsilon = 10^{-6}$}
\label{percentile_stocks_35_60_NN}
\end{subfigure}
\begin{subfigure}[t]{.33\linewidth}
\centering
\includegraphics[width=\linewidth]{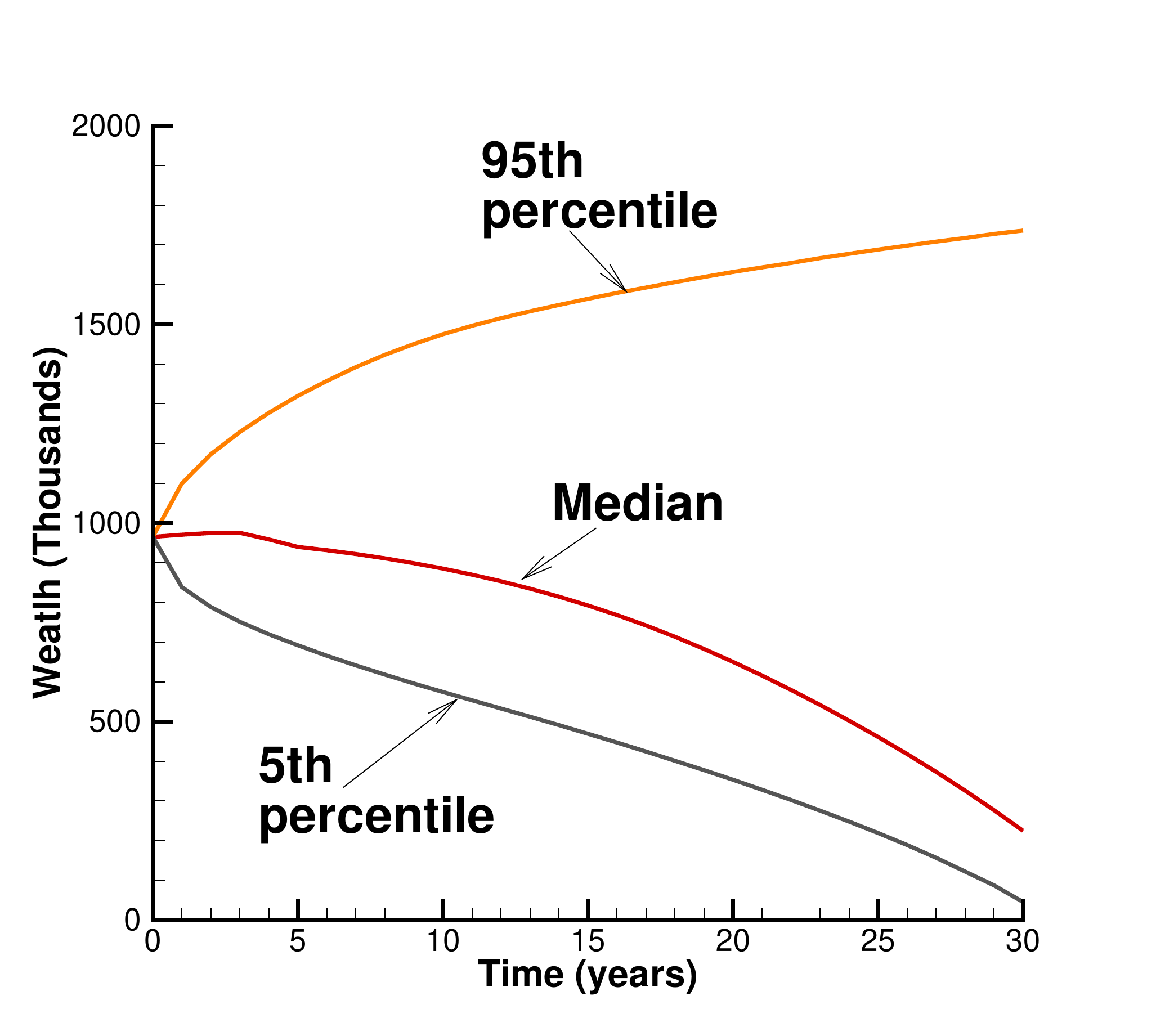}
\caption{Percentiles  wealth, NN Control, $\epsilon = 10^{-6}$}
\label{percentiles_wealth_35_60_NN}
\end{subfigure}
\begin{subfigure}[t]{.33\linewidth}
\centering
\includegraphics[width=\linewidth]{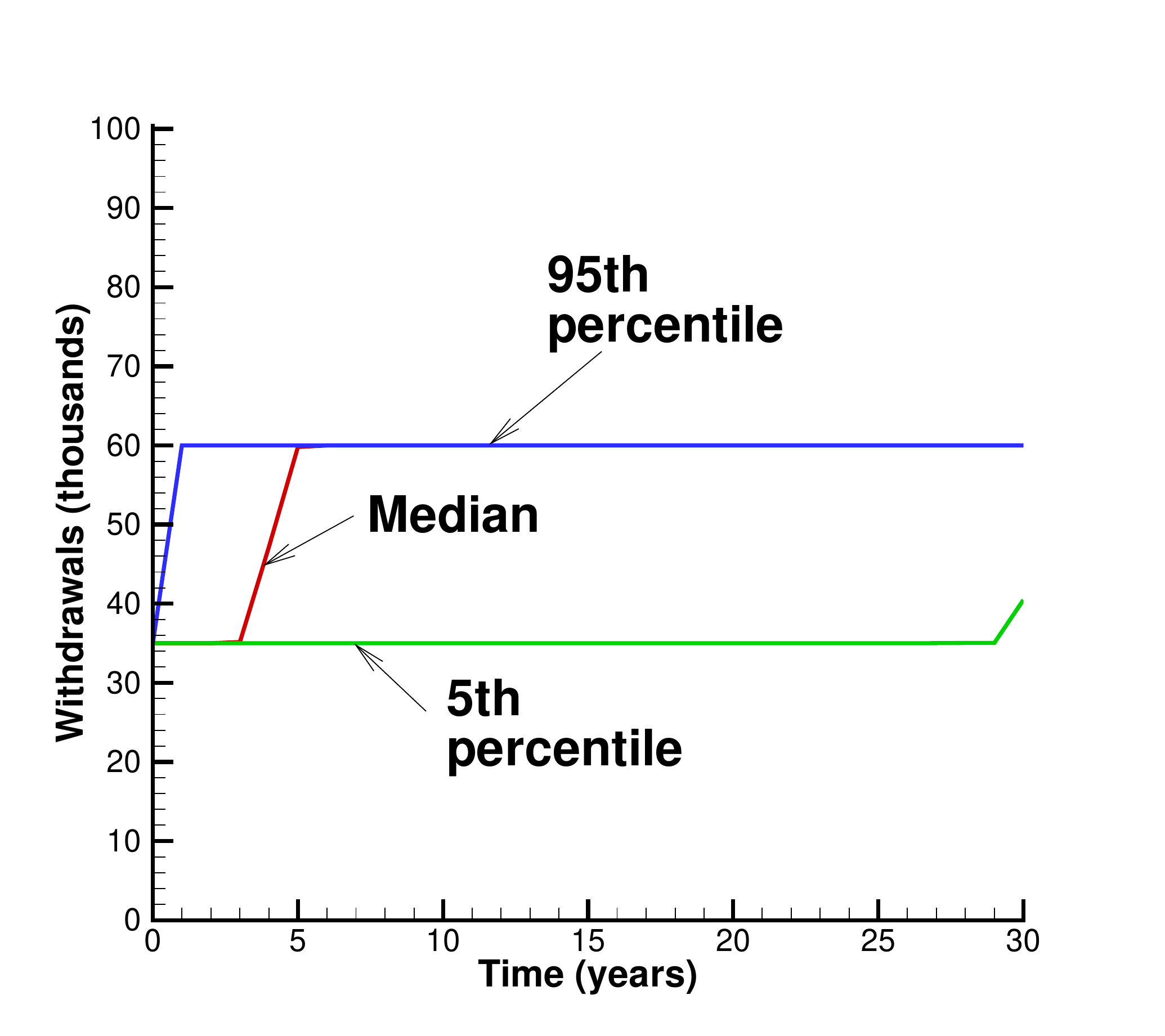}
\caption{Percentiles withdrawals, NN Control, $\epsilon = 10^{-6}$}
\label{percentiles_q_35_60_NN}
\end{subfigure}
}
\textbf{ }\par\medskip
\centering
\textbf{HJB Control Results (Positive and Negative Stabilization)}
\centerline{%
\begin{subfigure}[t]{.33\linewidth}
\centering
\includegraphics[width=\linewidth]{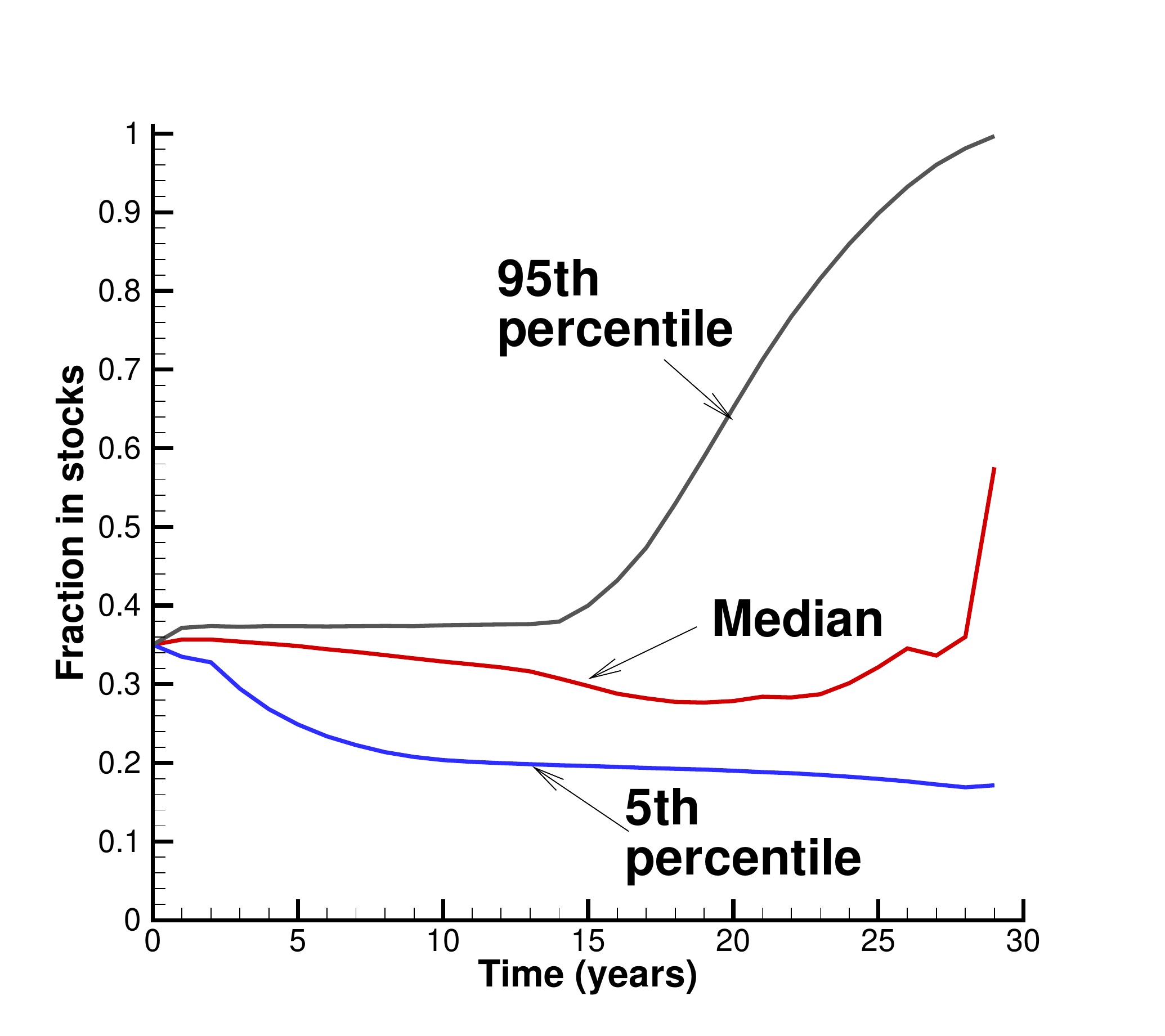}
\caption{Percentiles fraction in stocks, HJB Control, $\epsilon = 10^{-6}$}
\label{percentile_stocks_35_60_PIDE}
\end{subfigure}
\begin{subfigure}[t]{.33\linewidth}
\centering
\includegraphics[width=\linewidth]{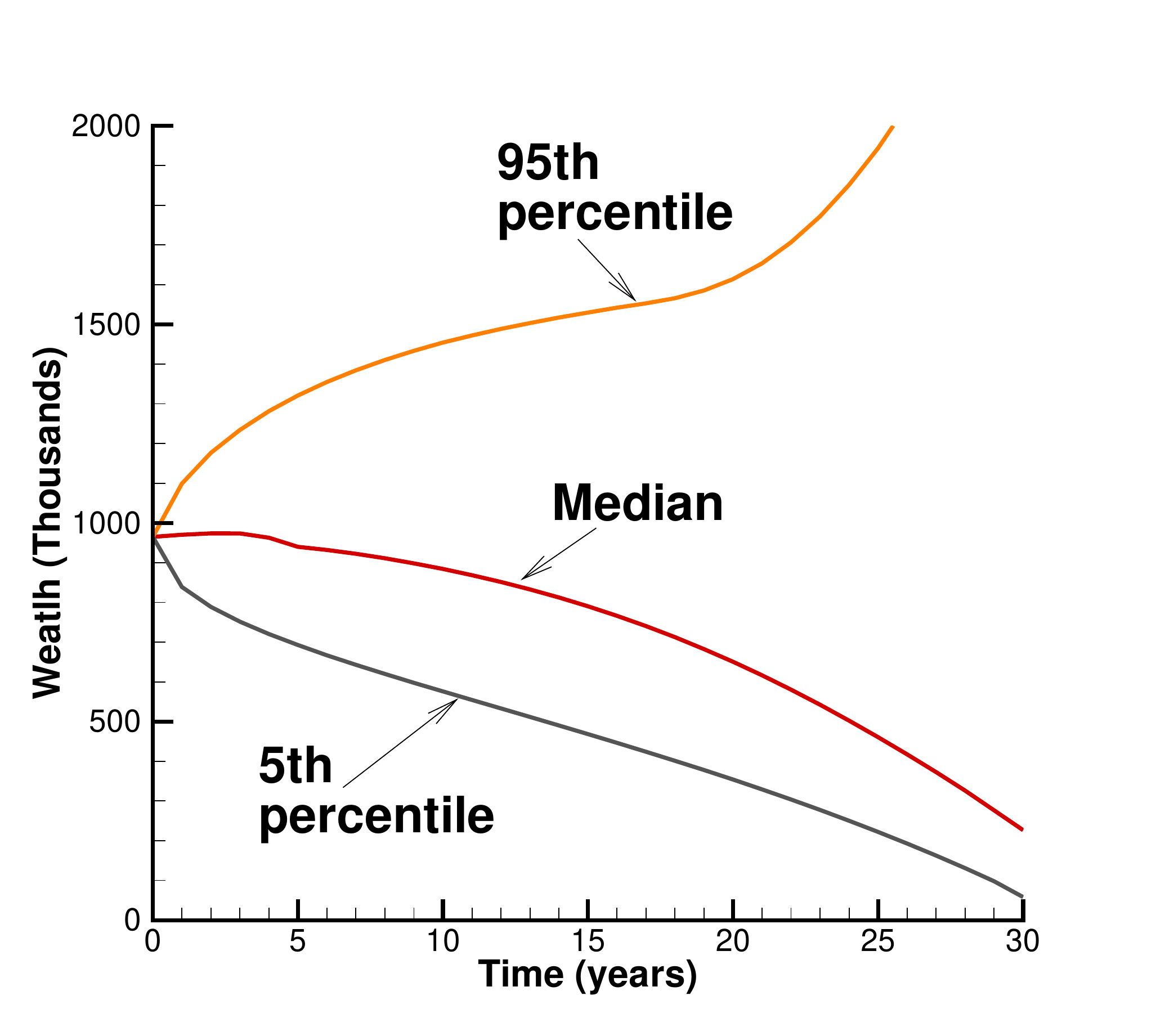}
\caption{Percentiles  wealth, HJB Control, $\epsilon = 10^{-6}$}
\label{percentiles_wealth_35_60_PIDE}
\end{subfigure}
\begin{subfigure}[t]{.33\linewidth}
\centering
\includegraphics[width=\linewidth]{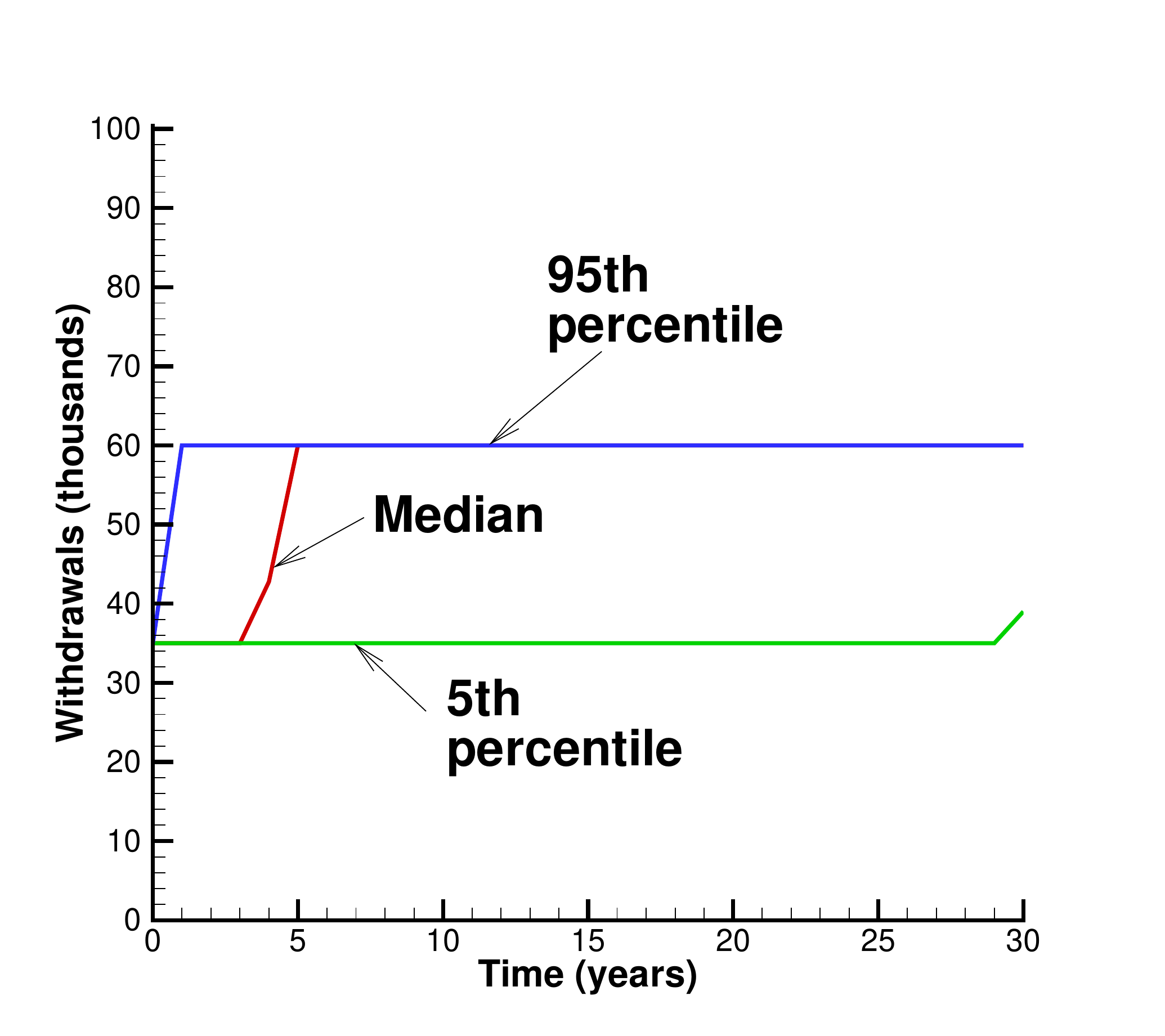}
\caption{Percentiles withdrawals, HJB Control, $\epsilon = 10^{-6}$}
\label{percentiles_q_35_60_PIDE}
\end{subfigure}
}
\centering
\centerline{%
\begin{subfigure}[t]{.33\linewidth}
\centering
\includegraphics[width=\linewidth]{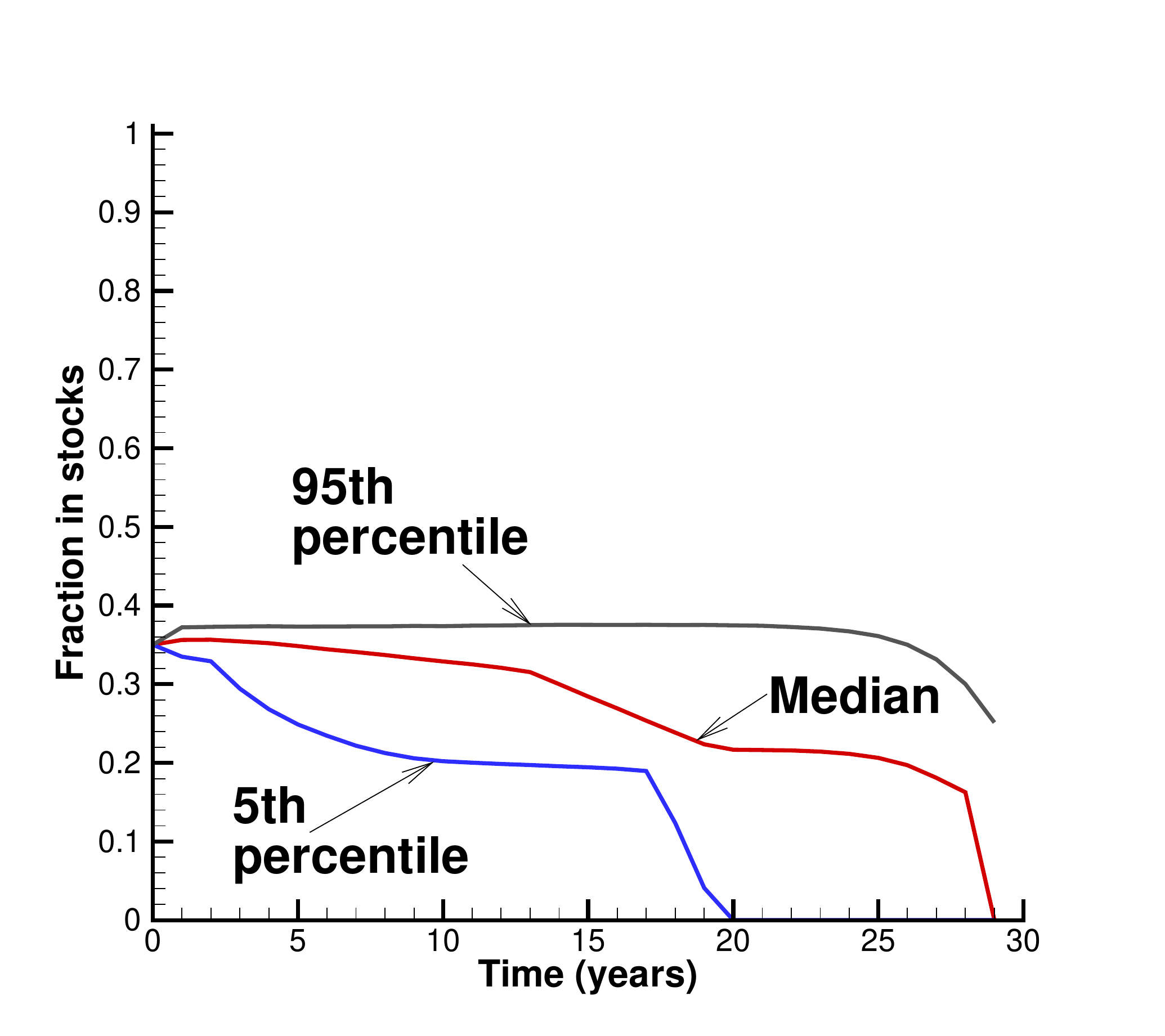}
\caption{Percentiles fraction in stocks, HJB Control, $\epsilon = -10^{-6}$}
\label{percentile_stocks_35_60_PIDE_negeps}
\end{subfigure}
\begin{subfigure}[t]{.33\linewidth}
\centering
\includegraphics[width=\linewidth]{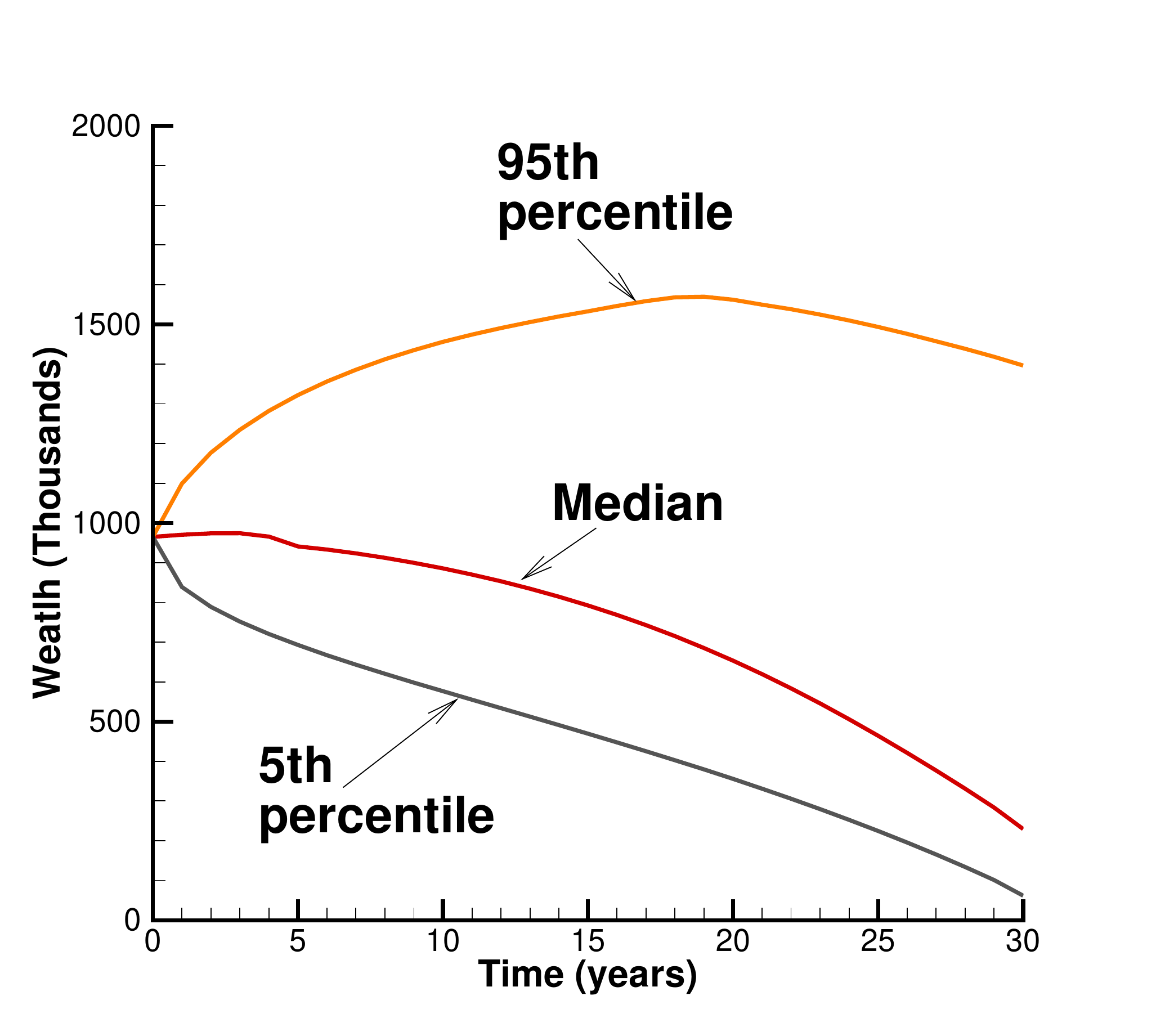}
\caption{Percentiles  wealth, HJB Control, $\epsilon = -10^{-6}$}
\label{percentiles_wealth_35_60_PIDE_negeps}
\end{subfigure}
\begin{subfigure}[t]{.33\linewidth}
\centering
\includegraphics[width=\linewidth]{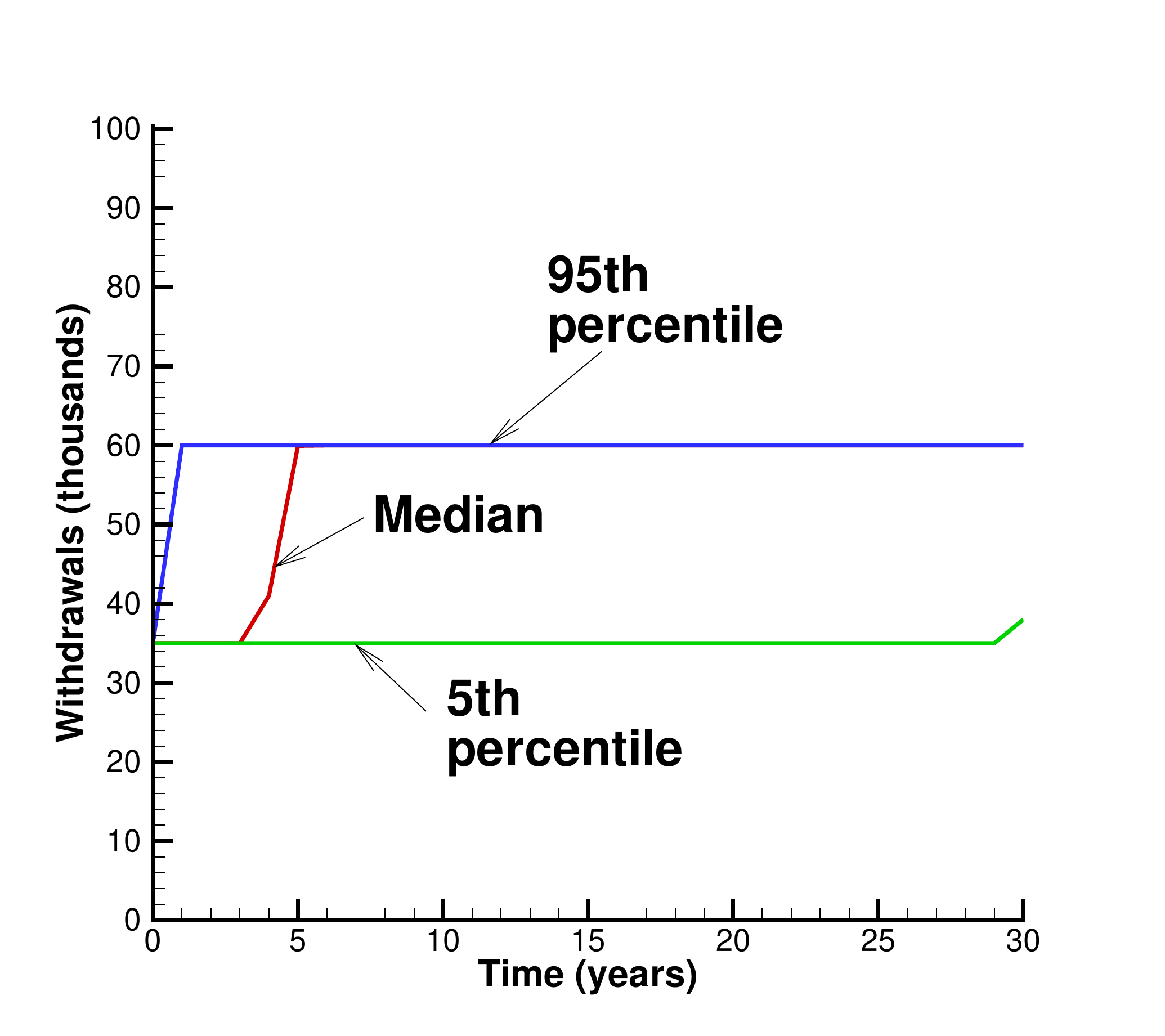}
\caption{Percentiles withdrawals, HJB Control, $\epsilon = -10^{-6}$}
\label{percentiles_q_35_60_PIDE_negeps}
\end{subfigure}
}

\caption{Scenario in Table \ref{base_case_1}.
NN and HJB controls computed from the problem (\ref{PCEE_a}).
Parameters based on the real CRSP index,
and real 10-year treasuries (see Table \ref{fit_params}).  
NN model trained on $2.56 \times 10^{5}$ observations of synthetic data.
HJB framework results from $2.56 \times 10^{6}$ observations of synthetic data.
$q_{min} = 35, q_{\max} = 60$, $\kappa = 1.0$. $W^* = 59.1$ for NN results. $W^* = 58.0$ for HJB results.
Units: thousands of dollars.
}
\label{percentiles_35_60}
\end{figure}

It is interesting to observe that the  proposed neural network framework  
is able to produce the {\em bang-bang} withdrawal control computed 
in \citet{forsyth:2022}, especially since we are using the continuous 
function $\hat{q}$ as an approximation.\footnote{Note that \citet{forsyth:2022} shows
that that in the continuous withdrawal limit, the withdrawal
control is bang-bang.  Our computed HJB results show that for discrete rebalancing,
the control appears to be bang-bang for all practical purposes.} A {\em bang-bang} control switches abruptly 
as shown here: the optimal strategy is to withdraw the minimum 
if the wealth is below a threshold, or else withdraw the maximum. 
As expected, the control threshold decreases as we move forward in time. 
We can see that the NN and HJB withdrawal controls behave very 
similarly at the 95th, 50th, and 5th percentiles of wealth 
(Figures \ref{percentiles_q_35_60_NN} and \ref{percentiles_q_35_60_PIDE}).  Essentially,
the optimal strategy withdraws at  either  $q_{\max}$ or $q_{\min}$, 
with a very small transition zone. This is in line with our expectations. 
By withdrawing less and investing more initially, 
the individual  decreases the chance of running out of savings.

We also note that the NN allocation control presents a small spread between the 5th and 95th percentile of the fraction in stocks (Figure \ref{percentile_stocks_35_60_NN}). In fact, the maximum stock allocation for the 95th percentile never 
exceeds 40\%, indicating that this is a stable low-risk strategy, which as we shall see, outperforms the 
\citet{Bengen1994} strategy.

\section{Model Robustness} \label{testing_section}
A common pitfall of neural networks is over-fitting to the training data. 
Neural networks that are over-fitted do not have the ability to generalize to previously unseen data. Since future asset 
return paths cannot be predicted, it is important to ascertain that the computed strategy is not overfitted to the training data and can perform well on unseen return paths. In this section, we demonstrate the robustness of the NN model's generated controls. 

We conduct three types of  robustness tests: (i) out-of-sample testing, 
(ii) out-of-distribution testing, and (iii) control sensitivity to training distribution.

\subsection{Out-of-sample testing}

Out-of-sample tests involve testing model performance on an unseen data 
set  sampled from the same distribution. In our case,  this means
training the NN on one set of  SDE paths sampled from the parametric model, and testing
on another set of paths generated using a different random seed.
We present the efficient frontier generated by computed 
controls on this new data set in Figure \ref{OOS_plot1}, 
which shows almost unchanged performance on the out-of-sample test set. 

\begin{figure}[htb!]
\centering
\includegraphics[width=0.4\linewidth]{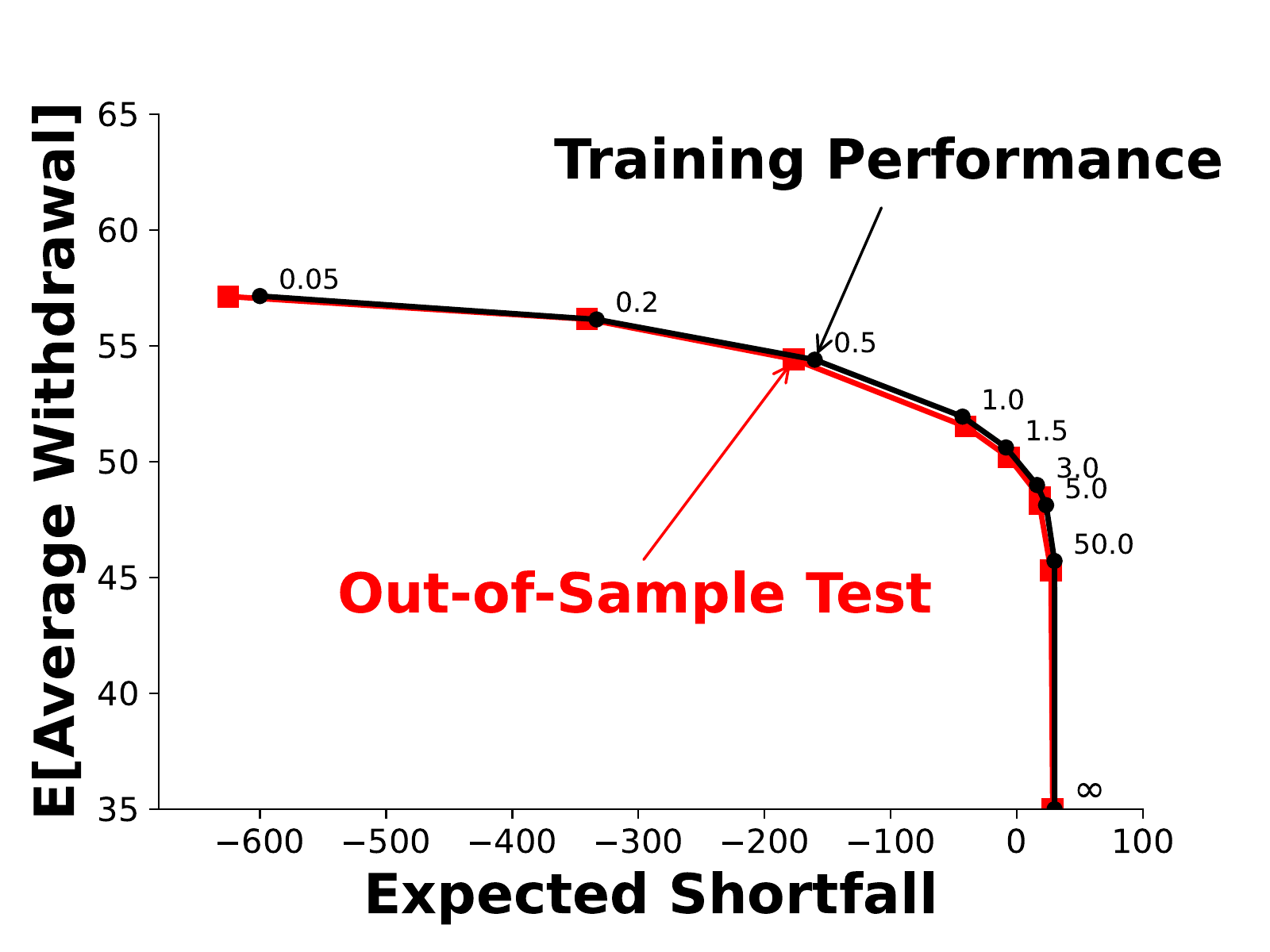}
\caption{Out-of-sample test. EW-ES frontiers, computed from the problem (\ref{PCEE_a}). Note: Scenario in Table \ref{base_case_1}. Comparison of NN training performance results vs. out-of-sample test. Both training and testing data are  $2.56 \times 10^5$ observations of synthetic data, generated with a different random seed. Parameters for synthetic data based on cap-weighted real CRSP, real 10 year treasuries (see Table \ref{fit_params}). 
$q_{min} = 35, q_{\max} = 60$.
$\epsilon = 10^{-6}$.
Units: thousands of dollars. Labels on nodes indicate $\kappa$ parameter values.}
\label{OOS_plot1}
\end{figure}

\subsection{Out-of-distribution testing}
    
Out-of-distribution testing involves evaluating the performance of 
the computed control on an entirely new data set sampled from a different distribution.
Specifically, test data is not generated from the parametric model used to produce training data, 
but is instead bootstrap resampled  from historical market returns via the method described in Section \ref{data_description_section}. 
We vary the expected block sizes to generate multiple testing data sets of $2.56 \times 10^5$ paths.

\begin{figure}
  \centering
  \includegraphics[width=0.4\linewidth]{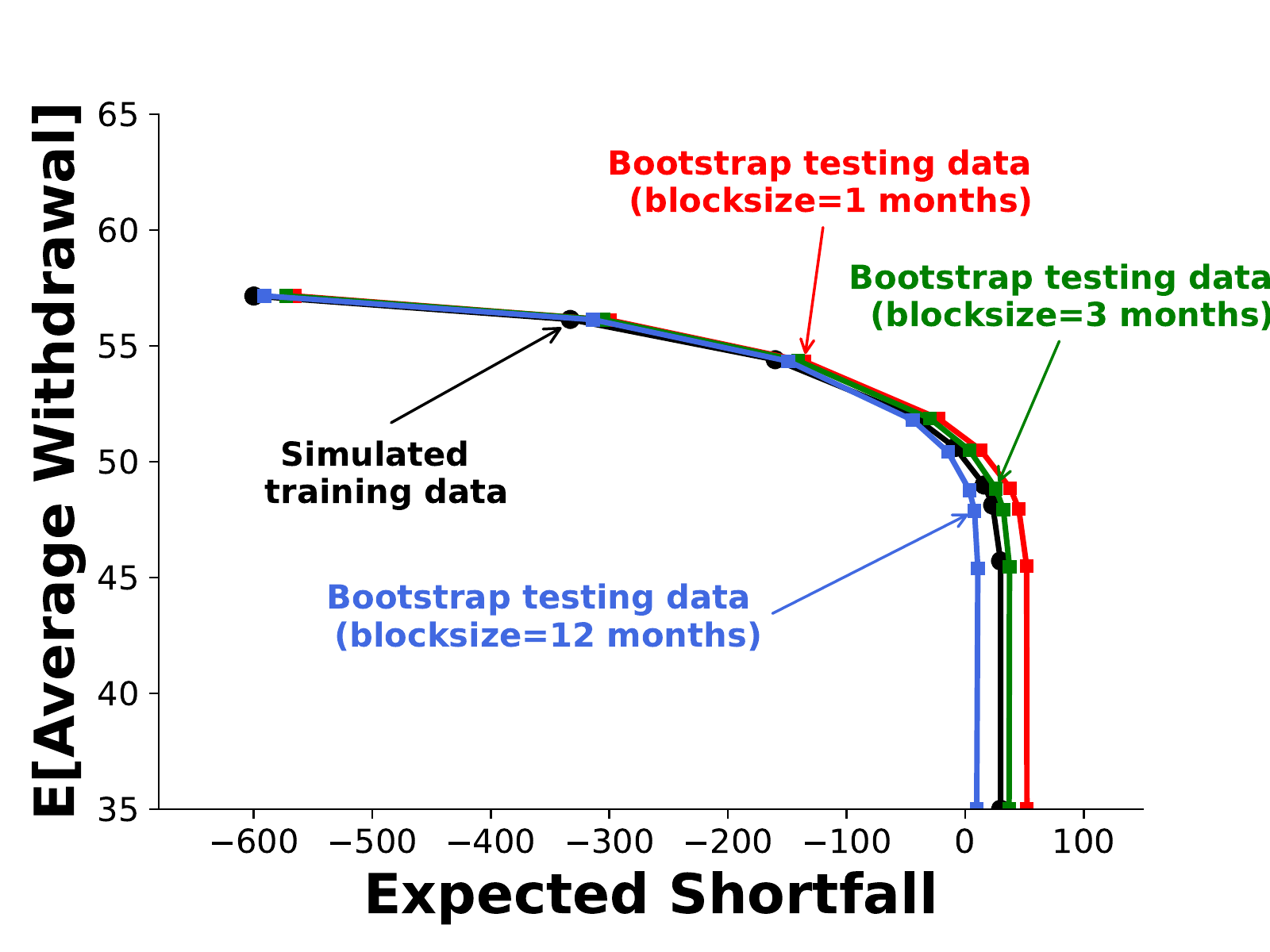}
  \caption{Out-of-distribution test. EW-ES frontiers of controls generated by NN model trained on $2.56 \times 10^{5}$ observations of synthetic data, tested on $2.56 \times 10^{5}$ observations of historical data with varying expected block sizes. Computed from  the problem (\ref{PCEE_a}). Note: Setup as in Table \ref{base_case_1}. Parameters based on real CRSP index and real 10-year U.S. Treasuries (see Table \ref{fit_params}). Historical data in range 1926:1-2019:12. Units: thousands of dollars. $q_{min} = 35; q_{max} = 60$. Simulated training data refers to Monte Carlo simulations
using the SDEs \eqref{jump_process_stock} and \eqref{jump_process_bond}.}
  \label{ood_test}
\end{figure}

In Figure \ref{ood_test}, 
we see that for each block size tested, the  efficient frontiers
are fairly close, indicating that the controls are relatively robust.
Note that the efficient frontiers for test performance in the historical
market with expected block size of 1 and 3 months plot slightly above the synthetic market frontier.  We conjecture that this
may be due to more pessimistic tail events
in the synthetic market.

The out-of-sample and out-of-distribution tests verify that the neural network is not over-fitting to the training data, and is generating an effective strategy, at least based on our block resampling data.

\subsection{Control sensitivity to training distribution} \label{control_sensitivity_subsection}

To  test the NN framework's adaptability to other training data sets, we train the NN framework on 
historical data (with expected block sizes of both 3 months and 12 months) 
and then test the resulting control on synthetic data. 
In Figure \ref{alt_training_test}, we compare the training performance and the test performance. 
The EW-ES frontiers for the test results on the synthetic data are very
close to the results on the bootstrap market data (training data set). 
This shows the NN framework's adaptability to use 
alternative data sets to learn, with the added advantage of not
being reliant on a parametric model, which is prone to miscalibration. 
Figure ~\ref{alt_training_test} also shows that, 
in all cases, in the synthetic or historical market, the EW-ES control significantly outperforms the 
Bengen \emph{4\% Rule} \footnote{The results for the Bengen strategy on the historical test data were computed with 
fixed  withdrawals of 4\% of initial
capital, adjusted for inflation.   We also
used a  constant allocation of 30\% in stocks for expected block size of 3 months, 
and 35\% in stocks for expected block size of 12 months. These were found to be the best 
performing constant allocations when paired with constant 4\% real withdrawals,
in terms of ES efficiency.} \citep{Bengen1994}. 
  
\begin{figure}[H]
\centering
    \begin{subfigure}[t]{.40\linewidth}
\centering
  \includegraphics[width=\linewidth]{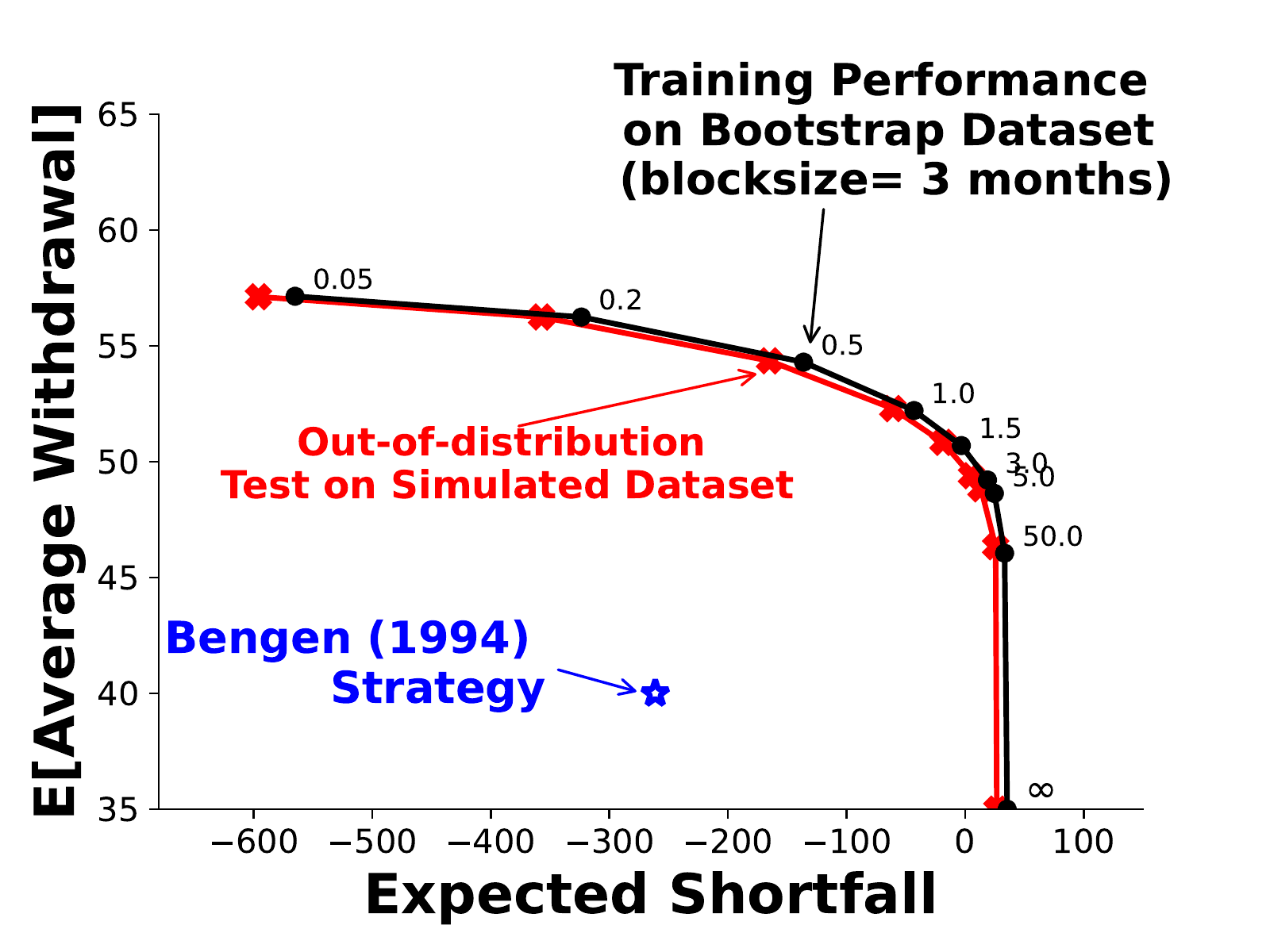}
  \caption{Historical training data, block size = 3 months}
  \end{subfigure}
  \hspace{.03\linewidth}
    \begin{subfigure}[t]{.40\linewidth}
  \centering
  \includegraphics[width=\linewidth]{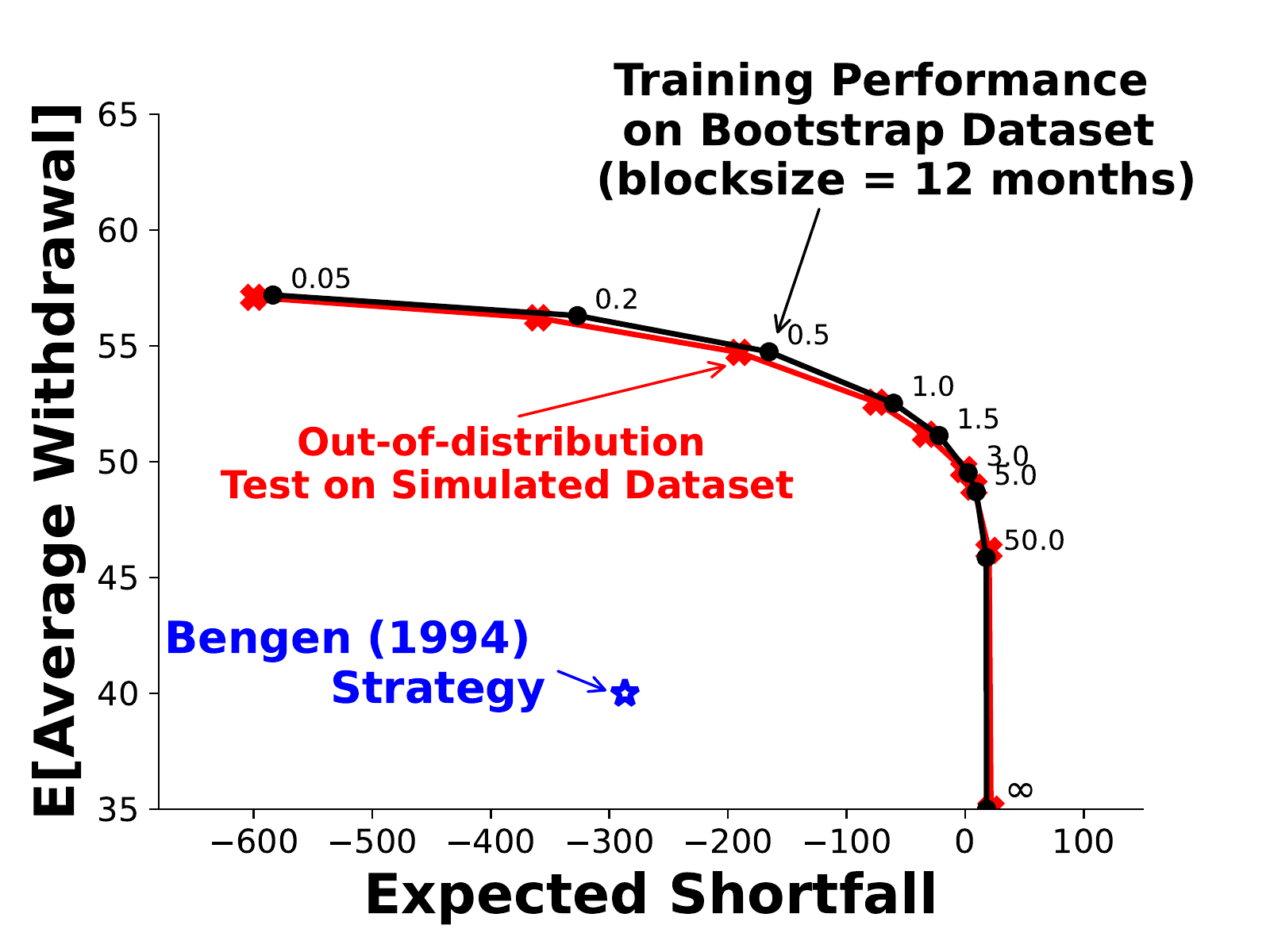}
  \caption{Historical training data, block size = 12 months}
  \end{subfigure}
  \caption{Training on historical data. EW-ES frontiers of controls generated by NN model trained on $2.56 \times 10^{5}$ observations of historical data with expected block sizes of a) 3 months and b) 12 months, 
each tested on $2.56 \times 10^{5}$ observations of synthetic data. 
Parameters based on real CRSP index and real 10-year U.S. Treasuries (see Table \ref{fit_params}). 
Historical data in range 1926:1-2019:12. Units: thousands of dollars. $q_{min} = 35; q_{max} = 60$.  The 
\citet{Bengen1994} results are based
on bootstrap resampling of the historical data.  Labels on nodes indicate $\kappa$ parameter values.
Simulated testing data refers to Monte Carlo simulations
using the SDEs \eqref{jump_process_stock} and \eqref{jump_process_bond}. $\epsilon = +10^{-6}$.
}
  \label{alt_training_test}
\end{figure}

\section{Conclusion}
In this paper, we proposed a { novel} neural network (NN)  architecture to efficiently and accurately
compute the  optimal decumulation strategy for retirees with DC pension plans. 
The stochastically constrained optimal control problem is solved based on  
a single standard unconstrained optimization, without using dynamic programming.

We began by highlighting the increasing prevalence of DC pension plans over traditional 
DB pension plans, and outlining the critical decumulation problem 
that faces DC plan investors.
There is an extensive literature  on devising
strategies for this problem. In particular, we examine a Hamilton-Jacobi-Bellman (HJB)
Partial Differential Equation (PDE)  based approach 
that can be shown to converge to an optimal solution for a dynamic withdrawal/allocation strategy. 
This provides an attractive balance of risk management and withdrawal efficiency for retirees. 
In this paper, we seek to build upon this approach by developing a new, 
more versatile framework using NNs to solve the decumulation problem. 

We conduct computational investigations  to demonstrate the accuracy and robustness of the proposed NN solution, 
utilizing the unique opportunity to compare NN solutions with the HJB results as a ground truth. 
Of particular noteworthiness is that the continuous function approximation from the NN framework is able 
to approximate a bang-bang control with high accuracy. 
We extend our experiments to establish the robustness of our approach, 
testing the NN control's performance on both synthetic and historical data sets. 

We demonstrate that the proposed NN framework produced solution accurately approximates  the ground truth solution. We also note the following advantages of the proposed NN framework:

\begin{enumerate} [label = (\roman*)]

    \item The NN method is data driven, and does not require postulating and calibrating a parametric
          model for market processes.

     \item The NN method directly estimates the low dimensional control by solving a single unconstrained optimization problem,
           avoiding the  problems associated with dynamic programming methods, which require estimating
           high dimensional conditional expectations (see \citet{beating_benchmark}). 

    \item The NN formulation maintains its simple structure (discussed in Section \ref{NN_framework}), immediately
       extendable to problems with more frequent rebalancing and/or withdrawal events. 
     In fact, the problem presented in \eqref{PCEE_a} requires each control NN to have only 
    two hidden layers for 30 rebalancing and withdrawal periods. 
    
\item The approximated control maintains continuity in time and/or space, provided it exists, 
   or otherwise provides a smooth approximation.  
   Continuity of the allocation control $p$ is an important practical consideration for any investment policy.

\end{enumerate}

Due to the ill-posedness of the stochastic optimal control problem in the region of 
high wealth near the end of the decumulation
 horizon, we observe that the NN allocation can appear to be very different from the HJB PDE solution. 
We note, however, that both strategies yield indistinguishable performance 
when assessed with the expected withdrawal and ES  reward-risk criteria.  In other words, these differences
hardly affect the objective function value, a weighted reward and risk value.  In the region of 
high wealth level near the end of the time horizon,
 the retiree is free to choose whether to invest
100\% in stocks or 100\% in bonds, since this has a negligible
effect on the objective function value (or reward-risk consideration).\footnote{This can be termed the {\em Warren Buffet}
effect.  Buffet is the fifth richest human being in the world.  He is 92 years old.
Buffet can choose any allocation strategy, and will never run out of cash.}

To conclude, the advantages of the NN framework make it a more versatile method,
compared to the solution of the HJB PDE.   We expect that the NN
approach  can handle problems of higher complexity,  e.g., involving a higher number of assets. 
In addition, the NN method can be applied to other proposed formulations for the retirement planning problem
(for example, see \citet{forsyth2022tontine}). We leave the extension of this methodology to future work.

\section{Acknowledgements}
Forsyth's work was supported by the Natural Sciences and Engineering Research Council of
Canada (NSERC) grant RGPIN-2017-03760.  Li's work was supported by the Natural Sciences and Engineering Research Council of
Canada (NSERC) grant RGPIN-2020-04331. The author's are grateful to P. van Staden for supplying the initial software library for
NN control problems.

\section{Conflicts of interest}
The authors have no conflicts of interest to report.


\appendix
\section*{Appendix}

\section{Induced Time Consistent Policy} \label{time_consistent_appendix}
In this section of the appendix, we review the concept of time consistency and relate its relevance to the $PCEE_{t_0}(\kappa)$ problem, (\ref{PCEE_a}).

Consider the optimal control $\mathcal{P}^*$ for problem (\ref{PCEE_a}), 

\begin{eqnarray} \label{opt_P_t0}
    (\mathcal{P}^*)^{t_0}(X(t_i^-), t_i)  ~;~ i=0, \ldots, M ~.
\end{eqnarray}
Equation (\ref{opt_P_t0}) can be interpreted as the optimal control for any time $t_i \geq t_0$, as a function of the state variables $X(t)$, as computed at $t_0$. 

Now consider if we were to solve the problem (\ref{PCEE_a}) starting at a later time $t_k, k > 0$. This optimal control starting at $t_k$ is denoted by:
\begin{eqnarray} 
    (\mathcal{P}^*)^{t_k}(X(t_i^-), t_i)  ~;~ i=k, \ldots, M \} ~.
\end{eqnarray}
In general, the solution of (\ref{PCEE_a}) computed at $t_k$ is not equivalent to the solution computed $t_0$:
\begin{eqnarray} 
    (\mathcal{P}^*)^{t_k}(X(t_i^-), t_i)  \neq  (\mathcal{P}^*)^{t_0}(X(t_i^-), t_i)  ~;~ i \ge k > 0.
\end{eqnarray}
This non-equivalence makes problem (\ref{PCEE_a}) \emph{time inconsistent}, implying that the investor will have the incentive to deviate from the control computed at time $t_0$ at later times. This type
of control is considered a {\em pre-commitment} 
control since the investor would need to commit to following the strategy at all times following $t_0$. 
Some authors describe pre-commitment controls as non-implementable because of the incentive to deviate.

In our case, however, the pre-commitment control from (\ref{PCEE_a}) can be shown to be identical to the time consistent control for an alternative version of the objective function. By holding $W^*$ fixed at the 
optimal value (at time zero), 
we can define the time consistent equivalent problem (TCEQ). 
Noting that the inner supremum in (\ref{PCEE_a}) is a continuous 
function of $W^*$, we define the optimal value of $W^*$ as 

\begin{eqnarray}
\mathcal{W}^*(s,b)  & = & \displaystyle \argmax_{W^*} \biggl\{
                      \sup_{\mathcal{P}_{0}\in\mathcal{A}} 
            \Biggl\{
               E_{\mathcal{P}_{0}}^{X_0^-,t_{0}^-}
           \Biggl[ ~\sum_{i=0}^{M} q_i ~  + ~
              \kappa \biggl( W^* + \frac{1}{\alpha} \min (W_T -W^*, 0) \biggr)
                \bigg\vert X(t_0^-) = (s,b)
                   ~\Biggr] \Biggr\} ~.
             \nonumber \\
            \label{pcee_argmax}
\end{eqnarray}
With a given initial wealth of $W_0^-$, this gives the following result from \citet{forsyth_2019_c}:
\begin{proposition}[Pre-commitment strategy equivalence to a time consistent
policy for an alternative objective function]

    The pre-commitment EW-ES strategy found by solving $J\left(s,b,  t_{0}^-\right)$ from (\ref{PCEE_a}), with fixed $W^* = \mathcal{W}^*$ from Equation \ref{pcee_argmax}, is identical to the time consistent strategy for the equivalent problem TCEQ (which has fixed $\mathcal{W}^*(0,W_0^-)$), with the following value function:

    \begin{eqnarray}
 \left(\mathit{TCEQ_{t_n}}\left(\kappa / \alpha  \right)\right):
     \nonumber \\
    \tilde{J}\left(s,b,  t_{n}^-\right)  
    & = &
            \sup_{\mathcal{P}_{n}\in\mathcal{A}}
        \Biggl\{
               E_{\mathcal{P}_{n}}^{X_n^-,t_{n}^-}
           \Biggl[ ~\sum_{i=n}^{M} q_i ~  + ~
               \frac{\kappa}{\alpha} \min (W_T -\mathcal{W}^*(0, W_0^-),0) 
                    \Biggr. \Biggr. 
                \bigg\vert X(t_n^-) = (s,b)
                   ~\Biggr] \Biggr\}.
                   \nonumber \\
                   \label{timec_equiv}
\end{eqnarray}
    
\end{proposition}

\begin{proof}
    This follows similar steps as in \citet{forsyth_2019_c}, proof of Proposition (6.2). 
\end{proof}

With fixed $W^*$, $\mathit{TCEQ_{t_n}}(\kappa / \alpha)$ 
is based on a target-based shortfall as its measure of risk, 
which is trivially time consistent. 
$W^*$ has the convenient interpretation of a disaster level of final wealth,
as specified at time zero.
Since the optimal controls for $PCEE_{t_0}(\kappa)$ and $\mathit{TCEQ_{t_n}}(\kappa / \alpha)$ 
are identical, we regard $\mathit{TCEQ_{t_n}}(\kappa / \alpha)$ as the EW-ES induced 
time consistent strategy \citep{Strub_2019_a}, which is implementable since 
the investor will have no incentive to deviate from a strategy computed at $t_0$ at later times. 

For further discussion concerning the relationship between pre-commitment, time consistent, and induced
time consistent strategies, we refer the reader to \citet{Bjork2010,Bjork2014,vigna:2014,Vigna2017,
Strub_2019_a,forsyth_2019_c,bjork_book_2021}.

\section{PIDE Between Rebalancing Times}\label{PIDE_all}
Applying Ito's Lemma for jump processes \citep{Cont_book}, using 
Equations (\ref{jump_process_stock}) and (\ref{jump_process_bond}) in Equation (\ref{expanded_6}) gives
\begin{eqnarray}
  & & V_t +  \frac{ (\sigma^s)^2 s^2}{2} V_{ss} +( \mu^s - \lambda_{\xi}^s \gamma_{\xi}^s) s V_s
       + \lambda_{\xi}^s \int_{-\infty}^{+\infty} V( e^ys, b, t) f^s(y)~dy 
      + \frac{ (\sigma^b)^2 b^2}{2} V_{bb} 
                  \nonumber \\
     & & ~~~     + ( \mu^b + \mu_c^b {\bf{1}}_{ \{ b < 0 \} } - \lambda_{\xi}^b \gamma_{\xi}^b) b V_b       
                  + \lambda_{\xi}^b \int_{-\infty}^{+\infty} V( s, e^yb, t) f^b(y)~dy 
       -( \lambda_{\xi}^s + \lambda_{\xi}^b )  V + \rho_{sb} \sigma^s \sigma^b s b V_{sb} 
       =0  ~,  \nonumber \\
  & & ~~~~~~~~~~~~~~~~~~~~~~~~~~~~~~~s \ge 0 ~.
      \label{expanded_7}
\end{eqnarray}
where the density functions $f^s(y), f^b(y)$ are as given in equation (\ref{eq:dist_stock}).

\section{Computational Details: Hamilton-Jacobi-Bellman (HJB) PDE  Framework} \label{appendix_PIDE}

For a detailed description of the numerical algorithm used to solve the HJB equation framework described in Section \ref{algo_section}, we refer the reader to \citet{forsyth:2022}. We summarize the method here. 

First, we solve the auxiliary problem (\ref{expanded_1}), with
fixed values of $W^*$, $\kappa$ and $\alpha$. The state space in $s > 0$ and $b >0$ is discretized using evenly spaced nodes in log space to create a grid to represent cases. A separate grid is created in a similar fashion to represent cases where wealth is negative. The Fourier methods discussed in \citet{forsythlabahn2017} are used to solve the PIDE representing market dynamics between rebalancing times. Both controls for withdrawal and allocation are discretized using equally spaced grids. The optimization problem (\ref{expanded_3}) is solved first for the allocation control by exhaustive search, storing the optimal for each discretized wealth node. The withdrawal control in (\ref{q_opt}) can then be solved in a similar fashion, using the previously stored allocation control to evaluate the right-hand side of (\ref{q_opt}). Linear interpolation is used where necessary. The stored controls are used to advance the solution in (\ref{expanded_5}). 

Since the numerical method just described assumes a constant $W^*$, an outer optimization step to find the optimal $W^*$ (candidate Value-at-Risk) is necessary. Given an approximate solution to (\ref{expanded_1}) at $t=0$, the full solution to 
$PCEE_{t_0}(\kappa)$ (\ref{PCEE_a}) is determined using Equation (\ref{expanded_equiv}). A coarse grid is used at first for an exhaustive search. This is then used as the starting point for a one-dimensional optimization algorithm on finer grids.

\section{Computational Details:  NN Framework} \label{appendix_nn}

\subsection{NN Optimization}
The NN framework, as described in Section \ref{NN_formulation_section} and illustrated in Figure \ref{NN_diagram}, was implemented using the PyTorch library \citep{NEURIPS2019_9015}. The withdrawal network $\hat{q}$, and allocation network $\hat{p}$ were both implemented with 2 hidden layers of 10 nodes each, with biases. Stochastic Gradient Descent \citep{ruder2016overview} was used in conjunction with the Adaptive Momentum optimization algorithm to train the NN framework \citep{kingma2017adam}. The NN parameters and auxiliary training parameter $W^*$ were trained with different initial learning rates. The same decay parameters and learning rate schedule were used. Weight decay 
($\ell_2$ penalty) was also employed to make training more stable. The training loop utilizes the auto-differentiation capabilities of the PyTorch library. Hyper-parameters used for NN training in this paper's experiments are given in Table \ref{nn_hyperparameters}. 

The training loop tracks the minimum loss function value as training progresses and selects the model that had given the optimal loss function value based on the entire training dataset by the end of the specified number of training epochs.

\subsection{Transfer learning between different $\kappa$ points}

 For high values of $\kappa$,  the objective function is  weighted more  towards optimizing ES (lower risk). In these cases,  optimal controls  are more difficult to compute.  
This is because the ES measure used (CVAR) is only affected by the sample paths below the $5^{th}$ percentile of terminal wealth, which are quite sparse. To overcome these training difficulties, we employ transfer learning \citep{tan2018survey} to improve training for the more difficult points on the efficient frontier. We begin training the model for the lowest $\kappa$ from a random initialization (`cold-start'), and then initialize the models for each increasing $\kappa$ with the model for the previous $\kappa$. Through numerical experiments, we found this method made training far more stable and less likely to 
terminate in local minima for higher values of $\kappa$. 

\subsection{Running minimum tracking}
The training loop tracks the minimum loss function value as training progresses and selects the model that had given the optimal loss function value based on the entire training dataset by the end of the specified number of training epochs. 

\begin{table}[h!]
\begin{center}
{\small
\begin{tabular}{lc} \toprule
NN framework hyper-parameter          & Value \\ \midrule
  Hidden layers per network   &  2 \\
  \# of nodes per hidden layer &  10 \\
  Nodes have biases &  True \\
  \# of iterations (\#itn)  & 50,000 \\
  SGD mini-batch size & 1,000 \\
  \# of training paths & $2.56 \times 10^5$ \\
  Optimizer & Adaptive Momentum \\
  Initial Adam learning rate for $(\bm{\theta}_q, \bm{\theta}_p)$ & 0.05 \\
  Initial Adam learning rate for $W^*$ & 0.04 \\
  Adam learning rate decay schedule &
  $[0.70 \times\text{\#itn}, 0.97 \times \text{\#itn}]$, $\gamma = 0.20$ \\
  Adam $\beta_1$ & 0.9 \\
  Adam $\beta_2$ & 0.998 \\
  Adam weight decay ($\ell_2$ Penalty) & 0.0001 \\
  Transfer Learning between $\kappa$ points & True \\
  Take running minimum as result & True \\
\bottomrule
\end{tabular}
}
\end{center}
\caption{Hyper-parameters used in training the NN framework for numerical experiments presented in this paper.  
\label{nn_hyperparameters}
}
\end{table}

\subsection{Standardization}
To improve learning for the neural network, we normalize the input wealth using means and standard deviations  of wealth samples from a  reference strategy. We use the constant withdrawal and allocation strategy defined in \citet{forsyth:2022} as the reference strategy with $2.56 \times 10^5$ simulated paths. Let $W^b_t$ denote the wealth {vector} at time $t$ based on simulations. Then $\bar{W}^b_t$ and $\sigma(W^b_t)$ denote the associated average wealth and standard deviation. Then we normalize the feature input to the neural network in the following way:

$$
\Tilde{W_t} = \frac{W_t - \bar{W}^b_t}{\sigma(W^b_t)}
$$
For the purpose of training the neural network, the values $\bar{W}^b_t$ and $\sigma(W^b_t)$ are just  constants, and we can use any reasonable values. This input feature normalization is done  for both withdrawal and allocation NNs.

In Section \ref{testing_section}, we show in out-of-sample and out-of-distribution tests that $\bar{W}^b_t$ and $\sigma(W^b_t)$ do not need to be related to the testing data as long as these are reasonable values. In Section \ref{NN_formulation_section}, when referring to $W$ as part of the input to the NN functions $\hat{q}$ and $\hat{p}$, we use the standardized $\Tilde{W}$ for computation.

\section{Model Calibrated from Market Data} \label{appendix_market_model}
Table \ref{fit_params} shows the calibrated model parameters for processes \eqref{jump_process_stock} and \eqref{jump_process_bond},
from \citet{forsyth:2022} using market data described in \S \ref{data_description_section}.  

{\small

\begin{table}[h!]
\begin{center}
\text{Calibrated Model Parameters}\\
\medskip
\begin{tabular}{cccccccc} \toprule
 CRSP & $\mu^s$ & $\sigma^s$ & $\lambda^s$ & $u^s$ &
  $\eta_1^s$ & $\eta_2^s$ & $\rho_{sb}$ \\ \midrule
       & 0.0877  & 0.1459&  0.3191  &  0.2333 & 4.3608 & 5.504 & 0.04554\\
 \midrule
 \midrule
10-year Treasury & $\mu^b$ & $\sigma^b$ & $\lambda^b$ & $u^b$ &
  $\eta_1^b$ & $\eta_2^b$ & $\rho_{sb}$ \\ \midrule
        & 0.0239 & 0.0538 & 0.3830 &  0.6111 &  16.19 & 17.27 & 0.04554\\
\bottomrule
\end{tabular}
\caption{Estimated annualized parameters for double exponential jump
diffusion model.  Value-weighted CRSP index, 10-year US treasury index
deflated by the CPI.  Sample period 1926:1 to 2019:12. 
}
\label{fit_params}
\end{center}
\end{table}
}

\section{Optimal expected block sizes: bootstrap resampling}
Table \ref{auto_blocksize} shows our estimates of the optimal block size using
the algorithm in \citet{politis2004,politis2009} using market data described in \S \ref{data_description_section}.
\begin{table}[h!]
\begin{center}

\text{Optimal expected block size for bootstrap resampling historical data} \\
\medskip
{\small
\begin{tabular}{lc} \toprule
Data series          & Optimal expected \\
                     & block size $\hat{b}$ (months) \\ \midrule
  Real 10-year Treasury index    &  4.2 \\
  Real CRSP value-weighted index &  3.1 \\
\bottomrule
\end{tabular}
}
\end{center}
\caption{Optimal expected blocksize $\hat{b}=1/v$ when the blocksize follows
a geometric distribution $Pr(b = k) = (1-v)^{k-1} v$. The algorithm in
\citet{politis2009} is used to determine $\hat{b}$.
Historical data range 1926:1-2019:12.
\label{auto_blocksize}
}
\end{table}

\section{Convergence Test: HJB Equation}
Table \ref{conservative_accuracy} shows a detailed convergence test for a single
point on the (EW, ES) frontier, using the PIDE method.
The controls are computed using the HJB PDE, and stored.  The stored
controls are then used in Monte Carlo simulations, which are used to
verify the PDE solution, and also generate various statistics of interest.

\begin{table}[h!]
\begin{center}
\medskip
{\small
\begin{tabular}{lcccc|cc} \toprule
\multicolumn{5}{c|}{Algorithm in Section \ref{algo_section} } & \multicolumn{2}{c}{Monte Carlo} \\ \midrule
Grid &   ES (5\%) & $E[ \sum_i q_i]/(M+1)$  &  Value Function &  $W^*$  &  ES (5\%) & $E[ \sum_i q_i]/(M+1)$  \\
\midrule 
$512 \times 512$   &-51.302  & 52.056 & 1.562430e+3  &   50.10      &         -45.936 &    52.07\\
$1024 \times 1024 $ & -46.239 &  52.049 &  1.567299e+3 &   52.47      &         -45.102 &    52.05\\
$2048 \times 2048 $ & -42.594 &  51.976 & 1.568671e+3  &  58.00       &         -42.623 &    51.97\\
$4096 \times 4096 $ &  -40.879 &  51.932&  1.569025e+3 &  61.08      &         -41.250 &   51.93\\
\midrule
\bottomrule
\end{tabular}
}
\caption{HJB equation convergence test,
real stock index: deflated real capitalization weighted CRSP, real bond index: deflated
ten year treasuries.  Scenario in Table \ref{base_case_1}.
Parameters in Table \ref{fit_params}.
The Monte Carlo method used
$2.56 \times 10^6$ simulations.
$\kappa = 1.0, \alpha = .05$.
Grid refers to the grid used in the Algorithm in Section \ref{algo_section}: $n_x \times n_b$, where $n_x$ is the number of nodes in the $\log s$ direction, and $n_b$ is the number of nodes in the $\log b$ direction.
Units: thousands of dollars (real).
$(M+1)$ is the total number of withdrawals. $M$ is the number of rebalancing dates.
$q_{\min} = 35.0$. $q_{\max} = 60$.
Algorithm in Section \ref{algo_section}.
\label{conservative_accuracy}}
\end{center}
\end{table}

\section{Detailed efficient frontier comparisons}
Table \ref{PIDE_ef_details} shows the detailed efficient frontier, computed
using the HJB equation method, using the $2048 \times 2048$ grid.
Table \ref{nn_ef_details} shows the efficient frontier computed from the NN framework. This should
be compared to Table \ref{PIDE_ef_details}.
Table \ref{appendix_obj_comparison} compares the objective function
values, at various points on the efficient frontier, for the HJB and
NN frameworks.

\begin{table}[hbt!]
\begin{center}
\text{Detailed Efficient Frontier: HJB Framework}\\
\medskip
\begin{tabular}{cccc} \toprule
$\kappa$ & ES (5\%) & $E[ \sum_i q_i]/(M+1)$  & $Median[W_T]$      \\ \midrule
  0.05 & -596.00  &    57.14  &             124.36\\
 0.2   & -334.29  &     56.17 &              92.99\\
 0.5 &   -148.99  &      54.25    &            111.20\\
 1.0 &     -42.62 &       51.97   &            227.84\\
  1.5 & -8.05    &   50.63      &        298.20 \\
  3.0 & 17.42     &    48.95     &          380.36\\
 5.0 &    24.09  &      48.12   &            414.60\\ 
 50.0 &   30.60   &     45.70    &           519.03\\
 $\infty$  &   31.00 &       35.00  &              1003.47\\
\bottomrule
\end{tabular}
\caption{Synthetic market results for HJB framework optimal strategies. Gives the detailed
results used to construct HJB efficient frontier in Figure \ref{EF_PIDE_nn}. Assumes the scenario given in Table~\ref{base_case_1}. Stock index: real capitalization weighted CRSP stocks;
bond index: ten year treasuries.  Parameters from Table \ref{fit_params}.
Units: thousands of dollars. Statistics based on $2.56 \times 10^6$ Monte Carlo simulation runs.
Control is computed using the Algorithm in Section \ref{algo_section}, ($2048 \times 2048$ grid) stored, and then used in the Monte Carlo simulations.
$q_{\min} = 35.0$, $q_{\max} = 60$.
$(M+1)$ is the number of withdrawals.
$M$ is the number of rebalancing dates.
$\epsilon = 10^{-6}$.
\label{PIDE_ef_details}
}
\end{center}
\end{table}

\begin{table}[h!]
\centering
\text{Detailed Efficient Frontier: NN Framework}\\
\medskip
\begin{tabular}{cccc}
\hline
$\kappa$ & ES (5\%) & $E[\sum_i q_i] / (M + 1)$ & $\text{\emph{Median}}[W_T]$ \\ \hline
0.05 & -599.81 & 57.15 & 106.23  \\
0.2  & -333.01 & 56.14 & 78.59   \\
0.5  & -160.14 & 54.40 & 105.05  \\
1    & -43.02  & 51.95 & 227.79  \\
1.5  & -8.57   & 50.62 & 302.17  \\
3    & 16.01   & 48.99 & 374.43   \\
5    & 23.20   & 48.13 & 425.13  \\
50   & 29.88   & 45.72 & 493.41  \\
$\infty$  & 29.90   & 35.00 & 947.60  \\ \hline
\end{tabular}
\caption{Synthetic market results for NN framework optimal strategies. Gives the detailed
results used to construct NN efficient frontier in Figure \ref{EF_PIDE_nn}. Assumes
the scenario given in Table~\ref{base_case_1}. Stock index: real capitalization weighted CRSP stocks;
bond index: ten year treasuries.  Parameters from Table \ref{fit_params}.
Units: thousands of dollars. Training performance statistics
based on $2.56 \times 10^5$ Monte Carlo simulation runs.
Control is computed using the algorithm in Section \ref{NN_formulation_section}.
$q_{\min} = 35.0$, $q_{\max} = 60$.
$(M+1)$ is the number of withdrawals.
$M$ is the number of rebalancing dates.
$\epsilon = 10^{-6}$.}
\label{nn_ef_details}
\end{table}

\begin{table}[h!]
\centering
\text{Objective Function Value Comparison: HJB Framework vs. NN Framework} \\
\medskip
\begin{tabular}{ccccc}
\hline
$\kappa$ & HJB equation      & NN      & \% difference \\
\hline
0.05              & 1741.54  & 1741.71 & 0.01\%        \\
0.2               & 1674.41 & 1673.81 & -0.04\%       \\
0.5               & 1607.26 & 1606.44 & -0.05\%       \\
1                 & 1568.45  & 1567.34 & -0.07\%       \\
1.5               & 1557.46 & 1556.22 & -0.08\%       \\
3                 & 1569.71  & 1566.86 & -0.18\%       \\
5                 & 1612.16  & 1607.86 & -0.27\%       \\
50                & 2946.70   & 2911.10  & -1.21\%      \\
 \hline
\end{tabular}
\caption{Objective function value comparison for the HJB equation and NN framework model results on range of $\kappa$ values. Objective function values for both frameworks computed according to $PCEE_{t0}(\kappa)$ (higher is better). 
Assuming the scenario given in Table~\ref{base_case_1}. 
Stock index: real capitalization weighted CRSP stocks;
bond index: ten year treasuries.  
Parameters from Table \ref{fit_params}.
HJB solution statistics based on $2.56 \times 10^6$ Monte Carlo simulation runs.
HJB control is computed using the Algorithm in Section \ref{algo_section}, ($2048 \times 2048$ grid) stored, and then used in the Monte Carlo simulations.
NN Training performance statistics based on $2.56 \times 10^5$ Monte Carlo simulation runs.
Control is computed using the NN framework in Section \ref{NN_formulation_section}.
$q_{\min} = 35.0$, $q_{\max} = 60$.
$(M+1)$ is the number of withdrawals.
$M$ is the number of rebalancing dates.
$\epsilon = 10^{-6}$.}
\label{appendix_obj_comparison}
\end{table}

\clearpage

%



\noindent \setlength{\bibsep}{1pt plus 0.3ex}
{\small
\bibliographystyle{mynatbib}
\bibliography{benchmark_NN_bibliography}

\begin{thebibliography}{55}
\expandafter\ifx\csname natexlab\endcsname\relax\def\natexlab#1{#1}\fi
\expandafter\ifx\csname url\endcsname\relax
  \def\url#1{\texttt{#1}}\fi
\expandafter\ifx\csname urlprefix\endcsname\relax\def\urlprefix{URL }\fi

\bibitem[{Anarkulova et~al.(2022)Anarkulova, Cederburg, and
  O'Doherty}]{anarkulova2022stocks}
Anarkulova, A., S.~Cederburg, and M.~S. O'Doherty (2022).
\newblock Stocks for the long run? evidence from a broad sample of developed
  markets.
\newblock \emph{Journal of Financial Economics} 143(1), 409--433.

\bibitem[{Bengen(1994)}]{Bengen1994}
Bengen, W. (1994).
\newblock Determining withdrawal rates using historical data.
\newblock \emph{Journal of Financial Planning} 7, 171--180.

\bibitem[{Bernhardt and Donnelly(2018)}]{bernhardt-donnelly:2018}
Bernhardt, T. and C.~Donnelly (2018).
\newblock Pension decumulation strategies: A state of the art report.
\newblock Technical Report, Risk Insight Lab, Heriot Watt University.

\bibitem[{Bjork et~al.(2021)Bjork, Khapko, and Murgoci}]{bjork_book_2021}
Bjork, T., M.~Khapko, and A.~Murgoci (2021).
\newblock \emph{Time inconsistent control theory with finance applications}.
\newblock Springer Finance, New York.

\bibitem[{Bjork and Murgoci(2010)}]{Bjork2010}
Bjork, T. and A.~Murgoci (2010).
\newblock A general theory of {Markovian} time inconsistent stochastic control
  problems.
\newblock SSRN 1694759.

\bibitem[{Bjork and Murgoci(2014)}]{Bjork2014}
Bjork, T. and A.~Murgoci (2014).
\newblock A theory of {Markovian} time inconsisent stochastic control in
  discrete time.
\newblock \emph{Finance and Stochastics} 18, 545--592.

\bibitem[{Boukherouaa et~al.(2021)Boukherouaa, AlAjmi, Deodoro, Farias, and
  Ravikumar}]{imf_ai_finance}
Boukherouaa, E.~B., K.~AlAjmi, J.~Deodoro, A.~Farias, and R.~Ravikumar (2021).
\newblock Powering the digital economy: Opportunities and risks of artificial
  intelligence in finance.
\newblock \emph{IMF Departmental Papers} 2021(024), A001.

\bibitem[{Buehler et~al.(2019)Buehler, Gonon, Teichmann, and
  Wood}]{buehler2019deep}
Buehler, H., L.~Gonon, J.~Teichmann, and B.~Wood (2019).
\newblock Deep hedging.
\newblock \emph{Quantitative Finance} 19(8), 1271--1291.

\bibitem[{Cogneau and Zakamouline(2013)}]{cogneau2013block}
Cogneau, P. and V.~Zakamouline (2013).
\newblock Block bootstrap methods and the choice of stocks for the long run.
\newblock \emph{Quantitative Finance} 13:9, 1443--1457.

\bibitem[{Cont and Mancini(2011)}]{contmancini2011}
Cont, R. and C.~Mancini (2011).
\newblock Nonparametric tests for pathwise properties of semimartingales.
\newblock \emph{Bernoulli} 17, 781--813.

\bibitem[{Dang and Forsyth(2014)}]{dang-forsyth:2014a}
Dang, D.-M. and P.~A. Forsyth (2014).
\newblock Continuous time mean-variance optimal portfolio allocation under jump
  diffusion: a numerical impulse control approach.
\newblock \emph{Numerical Methods for Partial Differential Equations} 30,
  664--698.

\bibitem[{Dang and Forsyth(2016)}]{Dang2015a}
Dang, D.-M. and P.~A. Forsyth (2016).
\newblock Better than pre-commitment mean-variance portfolio allocation
  strategies: a semi-self-financing {Hamilton-Jacobi-Bellman} equation
  approach.
\newblock \emph{European Journal of Operational Research} 250, 827--841.

\bibitem[{Dichtl et~al.(2016)Dichtl, Drobetz, and Wambach}]{dichtl2016testing}
Dichtl, H., W.~Drobetz, and M.~Wambach (2016).
\newblock Testing rebalancing strategies for stock-bond portfolios across
  different asset allocations.
\newblock \emph{Applied Economics} 48(9), 772--788.

\bibitem[{Forsyth and Labahn(2019)}]{forsythlabahn2017}
Forsyth, P. and G.~Labahn (2019).
\newblock $\epsilon-${Monotone Fourier} methods for optimal stochastic control
  in finance.
\newblock \emph{Journal of Computational Finance} 22:4, 25--71.

\bibitem[{Forsyth(2020)}]{forsyth_2019_c}
Forsyth, P.~A. (2020).
\newblock Multi-period mean {CVAR} asset allocation: Is it advantageous to be
  time consistent?
\newblock \emph{SIAM Journal on Financial Mathematics} 11:2, 358--384.

\bibitem[{Forsyth(2022)}]{forsyth:2022}
Forsyth, P.~A. (2022).
\newblock A stochastic control approach to defined contribution plan
  decumulation: The nastiest, hardest problem in finance.
\newblock \emph{North American Actuarial Journal} 26:2, 227--251.

\bibitem[{Forsyth and Vetzal(2019)}]{Forsyth_Vetzal_2019a}
Forsyth, P.~A. and K.~R. Vetzal (2019).
\newblock Optimal asset allocation for retirement savings: deterministic vs.
  time consistent adaptive strategies.
\newblock \emph{Applied Mathematical Finance} 26:1, 1--37.

\bibitem[{Forsyth et~al.(2022)Forsyth, Vetzal, and
  Westmacott}]{forsyth2022tontine}
Forsyth, P.~A., K.~R. Vetzal, and G.~Westmacott (2022).
\newblock Optimal performance of a tontine overlay subject to withdrawal
  constraints.
\newblock \emph{arXiv 2211.10509} .

\bibitem[{Han and E(2016)}]{han_weinan_stack}
Han, J. and W.~E (2016).
\newblock Deep learning approximation for stochastic control problems.
\newblock \emph{CoRR} abs/1611.07422.

\bibitem[{Homer and Sylla(2005)}]{Homer_rates}
Homer, S. and R.~Sylla (2005).
\newblock \emph{A History of Interest Rates}.
\newblock Wiley, New York.

\bibitem[{Hur{\'{e}} et~al.(2021)Hur{\'{e}}, Pham, Bachouch, and
  Langren{\'{e}}}]{hure_nn_rl}
Hur{\'{e}}, C., H.~Pham, A.~Bachouch, and N.~Langren{\'{e}} (2021).
\newblock Deep neural networks algorithms for stochastic control problems on
  finite horizon: Convergence analysis.
\newblock \emph{{SIAM} Journal on Numerical Analysis} 59:1, 525--557.

\bibitem[{Ismailov(2022)}]{Ismailov}
Ismailov, V. (2022).
\newblock A three layer neural network can represent any multivariate function
  ArXiv:2012.03016.

\bibitem[{Kingma and Ba(2014)}]{kingma2017adam}
Kingma, D. and J.~Ba (2014).
\newblock Adam: A method for stochastic optimization p. arXiv:1412.6980.

\bibitem[{Kou(2002)}]{kou:2002}
Kou, S.~G. (2002).
\newblock A jump-diffusion model for option pricing.
\newblock \emph{Management Science} 48, 1086--1101.

\bibitem[{Kou and Wang(2004)}]{Kou2004}
Kou, S.~G. and H.~Wang (2004).
\newblock Option pricing under a double exponential jump diffusion model.
\newblock \emph{Management Science} 50, 1178--1192.

\bibitem[{Lauri{\`e}re et~al.(2021)Lauri{\`e}re, Pironneau
  et~al.}]{lauriere2021performance}
Lauri{\`e}re, M., O.~Pironneau, et~al. (2021).
\newblock Performance of a markovian neural network versus dynamic programming
  on a fishing control problem.
\newblock \emph{arXiv preprint arXiv:2109.06856} .

\bibitem[{Lin et~al.(2015)Lin, MacMinn, and Tian}]{Lin_2015}
Lin, Y., R.~MacMinn, and R.~Tian (2015).
\newblock De-risking defined benefit plans.
\newblock \emph{Insurance: Mathematics and Economics} 63, 52--65.

\bibitem[{MacDonald et~al.(2013)MacDonald, Jones, Morrison, Brown, and
  Hardy}]{MacDonald2013}
MacDonald, B.-J., B.~Jones, R.~J. Morrison, R.~L. Brown, and M.~Hardy (2013).
\newblock Research and reality: A literature review on drawing down retirement
  financial savings.
\newblock \emph{North American Actuarial Journal} 17, 181--215.

\bibitem[{MacMinn et~al.(2014)MacMinn, Brockett, Wang, Lin, and
  Tian}]{mitchell_2014}
MacMinn, R., P.~Brockett, J.~Wang, Y.~Lin, and R.~Tian (2014).
\newblock The securitization of longevity risk and its implications for
  retirement security.
\newblock In O.~S. Mitchell, R.~Maurer, and P.~B. Hammond, eds.,
  \emph{Recreating Sustainable Retirement}, pp. 134--160. Oxford University
  Press, Oxford.

\bibitem[{Mancini(2009)}]{mancini2009}
Mancini, C. (2009).
\newblock Non-parametric threshold estimation models with stochastic diffusion
  coefficient and jumps.
\newblock \emph{Scandinavian Journal of Statistics} 36, 270--296.

\bibitem[{Marler and Arora(2004)}]{pareto_optimality}
Marler, R. and J.~Arora (2004).
\newblock Survey of multi-objective optimization methods for engineering.
\newblock \emph{Structural and Multidisciplinary Optimization} 26, 369--395.

\bibitem[{Ni et~al.(2022)Ni, Li, Forsyth, and Carroll}]{ni2022optimal}
Ni, C., Y.~Li, P.~Forsyth, and R.~Carroll (2022).
\newblock Optimal asset allocation for outperforming a stochastic benchmark
  target.
\newblock \emph{Quantitative Finance} 22:9, 1595--1626.

\bibitem[{Paszke et~al.(2019)Paszke, Gross, Massa, Lerer, Bradbury, Chanan,
  Killeen, Lin, Gimelshein, Antiga, Desmaison, Kopf, Yang, DeVito, Raison,
  Tejani, Chilamkurthy, Steiner, Fang, Bai, and Chintala}]{NEURIPS2019_9015}
Paszke, A., S.~Gross, F.~Massa, A.~Lerer, J.~Bradbury, G.~Chanan, T.~Killeen,
  Z.~Lin, N.~Gimelshein, L.~Antiga, A.~Desmaison, A.~Kopf, E.~Yang, Z.~DeVito,
  M.~Raison, A.~Tejani, S.~Chilamkurthy, B.~Steiner, L.~Fang, J.~Bai, and
  S.~Chintala (2019).
\newblock Pytorch: An imperative style, high-performance deep learning library.
\newblock In \emph{Advances in Neural Information Processing Systems 32}, pp.
  8024--8035. Curran Associates, Inc.

\bibitem[{Patton et~al.(2009)Patton, Politis, and White}]{politis2009}
Patton, A., D.~Politis, and H.~White (2009).
\newblock Correction to: automatic block-length selection for the dependent
  bootstrap.
\newblock \emph{Econometric Reviews} 28, 372--375.

\bibitem[{Pfau(2018)}]{Pfau_2018}
Pfau, W.~D. (2018).
\newblock An overview of retirement income planning.
\newblock \emph{Journal of Financial Counseling and Planning} 29:1, 114:120.

\bibitem[{Pfeiffer et~al.(2013)Pfeiffer, Salter, and Evensky}]{Pfeiffer_2013}
Pfeiffer, S., J.~R. Salter, and H.~E. Evensky (2013).
\newblock Increasing the sustainable withdrawal rate using the standby reverse
  mortgage.
\newblock \emph{Journal of Financial Planning} 26:12, 55--62.

\bibitem[{Politis and Romano(1994)}]{politis1994}
Politis, D. and J.~Romano (1994).
\newblock The stationary bootstrap.
\newblock \emph{Journal of the American Statistical Association} 89,
  1303--1313.

\bibitem[{Politis and White(2004)}]{politis2004}
Politis, D. and H.~White (2004).
\newblock Automatic block-length selection for the dependent bootstrap.
\newblock \emph{Econometric Reviews} 23, 53--70.

\bibitem[{Powell(2021)}]{Powell2021}
Powell, W.~B. (2021).
\newblock From reinforcement learning to optimal control: A unified framework
  for s equential decisions.
\newblock In K.~G. Vamvoudakis, Y.~Wan, F.~L. Lewis, and D.~Cansever, eds.,
  \emph{Handbook of Reinforcement Learning and Control}, pp. 29--74. Springer
  International Publishing.

\bibitem[{Ritholz(2017)}]{Sharpe2017}
Ritholz, B. (2017).
\newblock Tackling the `nastiest, hardest problem in finance'.
\newblock
  \url{www.bloomberg.com/view/articles/2017-06-05/tackling-the-nastiest-hardest-problem-in-finance}.

\bibitem[{Rockafellar and Uryasev(2000)}]{Uryasev_2000}
Rockafellar, R.~T. and S.~Uryasev (2000).
\newblock Optimization of conditional value-at-risk.
\newblock \emph{Journal of Risk} 2, 21--42.

\bibitem[{Ruder(2016)}]{ruder2016overview}
Ruder, S. (2016).
\newblock An overview of gradient descent optimization algorithms.
\newblock \emph{arXiv:1609.04747} .

\bibitem[{Scott et~al.(2009)Scott, Sharpe, and Watson}]{scott-jason-sharpe}
Scott, J., W.~Sharpe, and J.~Watson (2009).
\newblock The 4\% rule – at what price?
\newblock \emph{Journal of Investment Management} 7:3, 31--48.

\bibitem[{Scott and Cavaglia(2017)}]{scott2017wealth}
Scott, L. and S.~Cavaglia (2017).
\newblock A wealth management perspective on factor premia and the value of
  downside protection.
\newblock \emph{Journal of Portfolio Management} 43(3), 33--41.

\bibitem[{Shefrin and Thaler(1988)}]{shefrin-thaler2}
Shefrin, H.~M. and R.~H. Thaler (1988).
\newblock The behavioral life-cycle hypothesis.
\newblock \emph{Economic Inquiry} 26:4, 609--643.

\bibitem[{Simonian and Martirosyan(2022)}]{simonian2022sharpe}
Simonian, J. and A.~Martirosyan (2022).
\newblock Sharpe parity redux.
\newblock \emph{The Journal of Portfolio Management} 48:4, 183--193.

\bibitem[{Strub et~al.(2019)Strub, Li, and Cui}]{Strub_2019_a}
Strub, M., D.~Li, and X.~Cui (2019).
\newblock An enhanced mean-variance framework for robo-advising applications.
\newblock SSRN 3302111.

\bibitem[{Tan et~al.(2018)Tan, Sun, Kong, Zhang, Yang, and Liu}]{tan2018survey}
Tan, C., F.~Sun, T.~Kong, W.~Zhang, C.~Yang, and C.~Liu (2018).
\newblock A survey on deep transfer learning.
\newblock \emph{arXiv 1808.01974} .

\bibitem[{Tankov and Cont(2009)}]{Cont_book}
Tankov, P. and R.~Cont (2009).
\newblock \emph{Financial Modelling with Jump Processes}.
\newblock Chapman and Hall/CRC, New York.

\bibitem[{Tsang and Wong(2020)}]{tsang_wong_dpl}
Tsang, K.~H. and H.~Y. Wong (2020).
\newblock Deep-learning solution to portfolio selection with serially-dependent
  returns.
\newblock \emph{SIAM Journal on Financial Mathematics} 11:2, 593--619.

\bibitem[{{U.S. Bureau of Labor Statistics}(2022)}]{bls_statistics_2023}
{U.S. Bureau of Labor Statistics} (2022).
\newblock Employee benefits survey: Latest numbers.
\newblock \url{https://www.bls.gov/ebs/latest-numbers.htm}.

\bibitem[{van Staden et~al.(2023)van Staden, Forsyth, and
  Li}]{beating_benchmark}
van Staden, P., P.~Forsyth, and Y.~Li (2023).
\newblock Beating a benchmark: dynamic programming may not be the right
  numerical approach.
\newblock \emph{SIAM Journal on Financial Mathematics} 14:2, 407--451.

\bibitem[{Vigna(2014)}]{vigna:2014}
Vigna, E. (2014).
\newblock On efficiency of mean-variance based portfolio selection in defined
  contribution pension schemes.
\newblock \emph{Quantitative Finance} 14, 237--258.

\bibitem[{Vigna(2022)}]{Vigna2017}
Vigna, E. (2022).
\newblock Tail optimality and preferences consistency for intertemporal
  optimization problems.
\newblock \emph{SIAM Journal on Financial Mathematics} 13:1, 295--320.

\bibitem[{Williams and Kawashima(2023)}]{williams_kawashima}
Williams, R. and C.~Kawashima (2023).
\newblock Beyond the 4\% rule: How much can you spend in retirement?
  \url{https://www.schwab.com/learn/story/beyond-4-rule-how-much-can-you-spend-retirement}.

\end{thebibliography}
}

\end{document}